\documentclass[12pt]{article}
\usepackage{pifont}
\usepackage{mathrsfs}
\usepackage{graphicx,ifpdf}
\usepackage{amsmath}
\usepackage{amsfonts,amsthm,amssymb,mathrsfs,bbding}
\usepackage{graphics,multicol}
\usepackage{color,fourier-otf}
\usepackage{float}
\usepackage{enumerate}
\usepackage{enumitem}
\usepackage[hypertexnames=false]{hyperref}
\usepackage[numbers]{natbib} 
\usepackage{thm-restate}
\usepackage{indentfirst}
\usepackage{latexsym,bm}
\allowdisplaybreaks[4]
\evensidemargin=\oddsidemargin\topmargin=-1.5cm

\newtheorem{conj}{Conjecture}
\newtheorem{lem}{Lemma}[section]

\newtheorem{cor}{Corollary}[section]
\newtheorem{defi}{Definition}[section]

\newtheorem{claim}{Claim}

\addtocounter{section}{0}

\begin{document}
	
	\title{Short Brooms in Edge-chromatic Critical Graphs}

	\author{{\small Yonglei Chen,\ \ Yan Cao \footnote{Corresponding author. Email: ycao@dlut.edu.cn (Yan Cao);  ylchen981127@mail.dlut.edu.cn (Yonglei Chen).}}\\[2mm]
		{\footnotesize School of Mathematical Sciences, Dalian University of Technology, Dalian, Liaoning 116024, China}}
	
	\date{}
	
	\maketitle {\flushleft\large\bf Abstract.}
	This paper studies short brooms in edge-chromatic critical graphs. We prove that for any short broom in a $\Delta$-critical graph, at most one color is missing at more than one vertex. Moreover, this color (if exists) is missing at exactly two vertices. Applying this result, we verify the Vertex-splitting Conjecture for graphs with $\Delta \geq 2(n-1)/3$ and the Overfull Conjecture for $\Delta$-critical graphs satisfying $\Delta \geq (2n+5\delta-12)/3$.
	
	\vspace{1mm}
	
	
	\section{Introduction}
	In this paper, we only consider simple graphs unless specified as multigraphs. Denote by $V(G)$, $E(G)$, $\Delta(G)$ and $\delta(G)$ the vertex set, edge set, maximum degree, and minimum degree of $G$, respectively. For two integers $p$ and $q$, we denote by $[p, q]$ the set $\{i \in \mathbb{Z}\mid p\leq i \leq q\}$. 
	A \emph{proper $k$-edge-coloring} of a graph $G$ is a mapping $\varphi: E(G) \to [1, k]$ such that $\varphi(e) \neq \varphi(f)$ for any two adjacent edges $e$ and $f$. Denote by $\mathcal{C}^k(G)$ the set of all $k$-edge-colorings of $G$. The \emph{chromatic index} $\chi'(G)$ is the smallest integer $k$ such that $G$ admits a proper $k$-edge-coloring. A graph is \emph{$\Delta$-critical} if $\chi'(G) = \Delta(G) + 1$ and $\chi'(H) < \chi'(G)$ for every proper subgraph $H$ of $G$. We follow \cite{GECB} for terminology and notation not defined here.
	
	Let $G$ be a graph with $e \in E(G)$, and let $\varphi \in \mathcal{C}^k(G-e)$ for some integer $k \geq 0$. For a $v \in V(G)$, define the two color sets $\varphi(v) = \{\varphi(f) : f \in E(G) \text{ incident to } v\}$ and $\overline{\varphi}(v) = [1, k] \backslash \varphi(v)$. 
	We call $\varphi(v)$ the set of \textit{colors present} at $v$ and $\overline{\varphi}(v)$ the set of \textit{colors missing} at $v$. In the case where $v$ has exactly one missing color ($|\overline{\varphi}(v)| = 1$), we allow $\overline{\varphi}(v)$ to also represent that specific color. This notion extends to a vertex set $X \subseteq V(G)$ by defining $\overline{\varphi}(X) = \bigcup_{v \in X} \overline{\varphi}(v)$. A vertex set $X$ is called \textit{$\varphi$-elementary} if the missing color sets of its vertices are pairwise disjoint, i.e., $\overline{\varphi}(u) \cap \overline{\varphi}(v) = \emptyset$ for all distinct $u, v \in X$. When the coloring $\varphi$ is clear from the context, we may simply refer to $X$ as \textit{elementary}.
	

	The concept of the \textit{Vizing fan} was introduced independently by Vizing \cite{Vizing1964, Vizing1965} and Gupta \cite{Gupta1967} in the 1960s as a key tool to prove that $\chi'(G)\leq \Delta(G)+\mu(G)$. Building on this concept, Stiebitz et al. \cite{GECB} later introduced an extension known as the \textit{multifan}. A fundamental property shared by both structures is that their vertex sets are $\varphi$-elementary.
	In 1984, Kierstead \cite{Kierstead1984} introduced the \textit{Kierstead path} to give a strengthening of Vizing's result. He showed that for every graph $G$ with $\chi'(G) =k+1$, the vertex set of Kierstead path with respect to a critical edge $e\in E(G)$ and a $\varphi \in \mathcal{C}^k(G-e)$ is $\varphi$-elementary if $k \geq \Delta(G)+1$. Subsequent work showed that Kierstead's argument also applies when $k =\Delta(G)$, provided a certain degree condition \cite{GECB}.
	In 2000, Tashkinov \cite{Tashkinov2000} introduced the \textit{Tashkinov tree}, which generalized both previous structures. He proved that if $G$ is a multigraph such that $\chi'(G)= k+1 \geq \Delta+2$, $e$ is a critical edge of $G$ and $T$ is a Tashkinov tree with respect to $e$ and $\varphi \in \mathcal{C}^k(G-e)$, then $V(T)$ is $\varphi$-elementary. 
	To apply Tashkinov's idea on simple graphs, Chen, Chen and Zhao \cite{ChenHamiltonicity2017} introduced the \textit{short broom} structure in 2017 as follow.
	\begin{defi}
		Let $G$ be a $\Delta$-critical graph, $xy\in E(G)$ and $\varphi \in \mathcal{C}^{\Delta}(G-xy)$. A short broom with respect to $xy$ and $\varphi \in \mathcal{C}^{\Delta}(G-xy)$ is a sequence $B=(x, xy, y, yz, z,$ $zv_1, v_1,$ $zv_2, v_2,\ldots,$ $zv_p, v_p)$ such that for all $i \in \{1, 2, 3,\cdots, p\}$, $e_i = zv_i$ and $\varphi(zv_i) \in \overline{\varphi}(\{x, y, z,$ $v_1,\cdots, v_{i-1}\})$. If $(x, xy, y, yz, z,$ $e_i, v_i)$ is a Kierstead path with respect to $xy$ and $\varphi$ for any $i \in \{1, 2,$ $3,\cdots, p\}$, we call $B$ a simple broom.
	\end{defi}
	They also proved that $V(B)$ is elementary if $\min\{d(y),d(z)\}<\Delta$ and $|\overline{\varphi}(x)\cup\overline{\varphi}(y)|\leq 4$.
	Later, the condition that $|\overline{\varphi}(x)\cup\overline{\varphi}(y)|\leq 4$ was eliminated in \cite{Cao2019survey}.
	However, the question of whether $V(B)$ remains elementary when $d(y)=d(z)=\Delta$ has remained unresolved.
	This question is particularly intriguing in light of the known behavior of the short Kierstead path $K=(x, xy, y, yz, z, zv, v)$, for which $V(K)$ is not necessarily $\varphi$-elementary when both $y$ and $z$ are vertices of maximum degree \cite{GECB}. However, it does satisfy the property that $|\overline{\varphi}(v) \cap (\overline{\varphi}(x) \cup \overline{\varphi}(y))| \leq 1$. Motivated by this analogy, we investigate the more complex short broom structure and establish the following result.
	
	For a short broom $B$ under $\varphi$ and a color $\alpha \in [1, \Delta]$, define $m_{\varphi, B}(\alpha)$ as $|\{v \in V(B) \mid \alpha \in \overline{\varphi}(v)\}| - 1$ if $\alpha \in \overline{\varphi}(V(B))$, and as 0 otherwise.
	
	\begin{restatable}[Main Theorem]{thm}{broom}\label{ne-broom}
		Let $G$ be a $\Delta$-critical graph with an edge $xy\in E(G)$, and let $\varphi \in \mathcal{C}^{\Delta}(G-xy)$. Suppose $B = (x, xy,$ $y, yz, z, zv_1,$ $v_1,\ldots, zv_p, v_p)$ is a  short broom with respect to $xy$ and $\varphi$. Then 
		\begin{equation*}
			\sum_{\alpha \in [1, \Delta]} m_{\varphi, B}(\alpha) \leq 1.
		\end{equation*}
	\end{restatable}
	
	This inequality implies that at most one color is missing at more than one vertex in $V(B)$, and if such a color exists, it is missing at exactly two vertices. This property is analogous to the known behavior of the short Kierstead path under the same degree conditions. Together with the result in \cite{Cao2019survey}, we are now allowed to use the short broom structure in much more scenarios. In this paper, we provide two of its applications as follows.
	
	We first address the \textit{Vertex-splitting Conjecture}. A \textit{vertex-splitting} in $G^*$ at a vertex $v$ gives a new graph $G$ obtained by replacing $v$ with two new adjacent vertices $v_1$ and $v_2$ and partitioning the neighborhood $N_{G^*}(v)$ into two nonempty subsets that, respectively, serve as the set of neighbors of $v_1$ and $v_2$ from $V(G^*)$ in $G$. We say $G$ is obtained from $G^*$ by a vertex-splitting. Hilton and Zhao \cite{vertex-splittingconj} proposed the following conjecture in 1997:
	\begin{conj}[Vertex-splitting Conjecture]
		Let $G^*$ be an $n$-vertex connected class 1 $\Delta$-regular graph with $\Delta>\frac{n}{3}$. If $G$ is obtained from $G^*$ by a vertex-splitting, then $G$ is $\Delta$-critical.
	\end{conj}
	Hilton and Zhao \cite{vertex-splittingconj} initially verified the conjecture for graphs with $\Delta > (3\sqrt{7} - 1)n/7 \approx 0.82n$. Later, Song \cite{Song2002} confirmed its validity for a special class of graphs with $\Delta > n/2$. Recently, Cao et al. \cite{CCSVS2022} showed that this Conjecture is true if $\Delta \geq \frac{3(n-1)}{4}$. In this paper, we verify the conjecture for graphs G with $\Delta \geq \frac{2(n - 1)}{3}$ as follows.
	
	\begin{restatable}{thm}{splittingthm}\label{vertex-splitting thm}
		Let $G^*$ be an $n$-vertex connected class 1 $\Delta$-regular graph with $\Delta \geq \frac{2(n - 1)}{3}$. If $G$ is obtained from $G^*$ by a vertex-splitting, then $G$ is $\Delta$-critical.
	\end{restatable}
	
	In fact, the Vertex-splitting Conjecture is a necessary condition for the Overfull Conjecture. In \cite{vertex-splittingconj}, Hilton and Zhao proved that the latter implies the former. Vizing's Theorem \cite{Vizing1964} established that for any simple graph $G$, $\Delta(G) \leq \chi'(G) \leq \Delta(G)+1$. 
	Following Fiorini and Wilson \cite{Fiorini1977}, graphs for which $\chi'(G) = \Delta(G)$ are said to be class 1, and otherwise they are class 2. 
	This leads naturally to a question: which graphs belong to class 1 and which to class 2?
	Holyer \cite{Holyer1981} proved that determining this classification is NP-complete. However, the Overfull Conjecture implies that for graphs with $\Delta(G) > \frac{n}{3}$, this determination can be done in polynomial time.
	A graph $G$ is \textit{overfull} if $|E(G)| > \Delta(G) \lfloor \frac{n}{2} \rfloor$. Clearly, overfull graphs must be class 2. The Overfull Conjecture, proposed by Chetwynd and Hilton \cite{ChetwyndHiltonStar1986, ChetwyndHiltonStar1988}, can be stated as follows. 
	\begin{conj}[Overfull Conjecture]
		If $G$ is a class 2 graph with $\Delta(G)>\frac{n}{3}$, then $G$ contains an overfull subgraph $H$ with $\Delta(H)=\Delta(G)$.
	\end{conj}
	
	They proved that the conjecture holds for all graphs with $\Delta \geq n-3$. In 2004, Plantholt \cite{plantholt2004overfull} showed that for graphs of even order, if $\delta(G) \geq \sqrt{7}n/3 \approx 0.8819n$, then $G$ satisfies the Overfull Conjecture.
	In \cite{Bongard2003}, Bongard et al. proved that the Overfull Conjecture holds under certain conditions  for graphs of odd order.
	Cao, Chen, Jing and Shan \cite{caoovf2022} provided a result for critical graphs in terms of a relation between maximum and minimum degree: if $G$ is $\Delta$-critical and satisfies $\Delta(G) - \frac{7\delta(G)}{4} \geq \frac{3n - 17}{4}$, then $G$ is overfull. In \cite{Shan2024}, Shan proved that for any $0\leq \varepsilon\leq \frac{1}{14}$, there exists a positive integer $n_0$ such that if $G$ is a graph on $n\geq n_0$ vertices with $\Delta(G)\geq(1-\varepsilon)n$, then the Overfull Conjecture holds for $G$. In this paper, we verify the conjecture for graphs G with $\Delta(G) \geq \frac{2n+5\delta(G)-12}{3}$ as follows.
	
	\begin{restatable}{thm}{overfullthm}\label{overfull thm}
		Let $G$ be a $\Delta$-critical graph of order $n$. If $\Delta(G) \geq \frac{2n+5\delta(G)-12}{3}$, then $G$ is overfull.
	\end{restatable}
	
	
	\section{Proof of the Main Theorem}
	
	\subsection{Definitions and Preliminary Results}
	
	
	This section outlines the foundational concepts and prior results essential to our proof. 
	For a graph $G$ and an edge $e \in E(G)$, denote by $G-e$ the graph obtained by deleting $e$ from $G$. A vertex is called a \emph{$k$-vertex} if its degree equals $k$, and a neighbor of vertex $v$ is a \emph{$k$-neighbor} if it is a $k$-vertex in $G$.
	%
	Let $G$ be a class 2 graph with $\Delta(G) = \Delta$. An edge $e$ in $G$ is called \textit{critical} if $\chi'(G-e)<\chi'(G)$. A connected class 2 graph is said to be \textit{$\Delta$-critical} if every edge is critical. 
	
	\begin{lem}[Vizing's Adjacency Lemma (VAL), \cite{VizingCritical1965}]\label{VAL}
		Let $G$ be a class 2 graph with maximum degree $\Delta$. If $e=xy$ is a critical edge of $G$, then $x$ has at least $\Delta-d_G(y)+1$ $\Delta$-neighbors in $V(G)\backslash\{y\}$.
	\end{lem}
	
	A key construct in edge coloring is the \textit{Kempe chain}. Let $G$ be a graph with $e \in E(G)$, and let $\varphi \in \mathcal{C}^k(G-e)$ for some integer $k \geq 0$. For five distinct colors $\alpha, \beta,$ $\zeta, \gamma,$ $\tau \in [1, k]$, an $(\alpha,\beta)$-chain is a component of the subgraph of $G$ induced by edges colored either $\alpha$ or $\beta$. Such a component is necessarily a path or an even cycle. 
	If an $(\alpha,\beta)$-chain $P$ is a path with an endvertex $x$, we use the notation $P_x(\alpha, \beta, \varphi)$ to stress this endpoint or simply $P_x(\alpha, \beta)$ when $\varphi$ is understood. For a vertex $u$ and an edge $uv$ contained in $P_x(\alpha, \beta, \varphi)$, we write $u \in P_x(\alpha, \beta, \varphi)$ and $uv \in P_x(\alpha, \beta, \varphi)$. If $u, v \in P_x(\alpha, \beta, \varphi)$ such that $u$ lies between $x$ and $v$ on $P$, then we say that $P_x(\alpha, \beta, \varphi)$ \textit{meets $u$ before $v$}. 
	If an $(\alpha,\beta)$-chain $P$ is a path containing both vertices $x$ and $y$, we denote by $P_{[x,y]}(\alpha, \beta, \varphi)$ the subchain with endvertices $x$ and $y$.  
	
	The operation of swapping the colors $\alpha$ and $\beta$ on all edges of an $(\alpha,\beta)$-chain $C$ is a \textit{Kempe change}, and the resulting new $k$-coloring is denoted $\varphi_1 = \varphi/C$. When the chain $C$ is a path starting at $x$, this operation is equivalently described as doing an \textit{$(\alpha,\beta)$-swap at $x$}. By convention, an $(\alpha,\alpha)$-swap at $x$ leaves the coloring unchanged.
	This concept extends to sequences of Kempe changes. Suppose $\beta_0 \in \overline{\varphi}(x)$ and $\beta_1, \ldots, \beta_t \in \varphi(x)$ for some integer $t \geq 1$. 
	A $(\beta_0,\beta_1)-(\beta_1,\beta_2)-\cdots-(\beta_{t-1},\beta_t)$-\textit{swap} at $x$ consists of $t$ consecutive Kempe changes: defining $\varphi_0 = \varphi$, we set $\varphi_i = \varphi_{i-1}/P_x(\beta_{i-1}, \beta_i, \varphi_{i-1})$ for each $i \in [1,t]$.
	If an edge $uv$ currently has color $\alpha$, the notation $uv: \alpha \rightarrow \beta$ indicates recoloring $uv$ with color $\beta$. 
	
	To compactly represent sequences of operations, we employ a two-row matrix notation. For example, the matrix
	\[
	\begin{bmatrix}
		P_{[a,b]}(\alpha,\beta,\varphi) &  rs & xy\\
		\alpha/\beta & \gamma \to \tau & \zeta
	\end{bmatrix}
	\]
	indicates three consecutive operations:
	\begin{itemize}[leftmargin=*, labelindent=1cm]
		\item[Step 1.] Exchange colors $\alpha$ and $\beta$ on the $(\alpha,\beta)$-subchain $P_{[a,b]}(\alpha,\beta,\varphi)$;
		\item[Step 2.] Recolor $rs$ from $\gamma$ to $\tau$ under the coloring obtained from Step 1;
		\item[Step 3.] Color $xy$ by $\zeta$ under the coloring obtained from Step 2.
	\end{itemize}
	
	The connectivity within these chains defines a relation between vertices. Two vertices $x, y \in V(G)$ are said to be \textit{$(\alpha,\beta)$-linked} under $\varphi$ if they are in the same $(\alpha,\beta)$-chain; otherwise, they are \textit{$(\alpha,\beta)$-unlinked}. When the coloring $\varphi$ is implicitly understood, we may omit the reference to it and state directly that $x$ and $y$ are $(\alpha,\beta)$-linked or unlinked.
	
	\subsection{Multifan, Kierstead Path and Short Broom}
	
	The fan argument, introduced by Vizing \cite{Vizing1964, Vizing1965}, was later generalized to the concept of a multifan \cite{GECB}.
	
	
	\begin{defi}
		Let $G$ be a graph, $e=rs_1 \in E(G)$ and $\varphi \in \mathcal{C}^k(G-e)$ for some integer $k \geq 0$. A multifan centered at $r$ with respect to $e$ and $\varphi$ is a sequence $F_{\varphi}(r, s_1: s_p):=(r, rs_1, s_1, rs_2, s_2, \ldots,$ $rs_p, s_p)$ with $p \geq 1$ consisting of distinct vertices $r, s_1, s_2, \ldots, s_p$ and distinct edges $rs_1,$ $rs_2, \ldots, rs_p$ satisfying the following condition:
		\begin{itemize}
			\item[(F1)] For every edge $rs_i$ with $i \in [2, p]$, there exists $j \in [1, i-1]$ such that $\varphi(rs_i) \in \overline{\varphi}(s_j)$.
		\end{itemize}
	\end{defi}
	
	We will simply denote a multifan $F_{\varphi}(r, s_1: s_p)$ by $F$ if $\varphi$ and the vertices and edges in $F_{\varphi}(r, s_1: s_p)$ are clear. The following result regarding a multifan can be found in [\cite{GECB}, Theorem 2.1].
	
	\begin{lem}[\cite{GECB}, Theorem 2.1]\label{multifan}
		Let $G$ be a class 2 graph and $F_{\varphi}(r, s_1: s_p)$ be a multifan with respect to a critical edge $e=rs_1$ and a coloring $\varphi \in \mathcal{C}^{\Delta}(G-e)$. Then the following statements hold.
		\begin{itemize}
			\item[(a)] $V(F)$ is $\varphi$-elementary.
			\item[(b)] Let $\alpha \in \overline{\varphi}(r)$. Then for every $i \in [1, p]$ and $\beta \in \overline{\varphi}(s_i)$, $r$ and $s_i$ are $(\alpha, \beta)$-linked with respect to $\varphi$.
		\end{itemize}
	\end{lem}
	
	\begin{defi}
		Let $G$ be a graph, $e=v_0v_1 \in E(G)$, and $\varphi \in \mathcal{C}^k(G-e)$ for some integer $k \geq 0$. A Kierstead path with respect to $e$ and $\varphi$ is a sequence $K=(v_0, v_0v_1, v_1, v_1v_2, v_2, \ldots,$ $v_{p-1},$ $v_{p-1}v_p, v_p)$ with $p \geq 1$ consisting of distinct vertices $v_0, v_1, \ldots, v_p$ and distinct edges $v_0v_1, v_1v_2, \ldots,$ $v_{p-1}v_p$ satisfying the following condition:
		\begin{itemize}
			\item[(K1)] For every edge $v_iv_{i+1}$ with $i \in [1, p-1]$, there exists $j \in [0, i-1]$ such that $\varphi(v_iv_{i+1})$ $\in \overline{\varphi}(v_j)$.
		\end{itemize}
	\end{defi}
	
	Clearly a Kierstead path with at most 3 vertices is a multifan. We consider Kierstead paths with 4 vertices. The result below was proved in Theorem 3.3 from \cite{GECB}.
	
	\begin{lem}[\cite{GECB}]\label{Kierstead path}
		Let $G$ be a class 2 graph, $e=v_0v_1 \in E(G)$ be a critical edge, and $\varphi \in \mathcal{C}^{\Delta}(G-e)$. If $K=(v_0, v_0v_1, v_1, v_1v_2, v_2, v_2v_3, v_3)$ is a Kierstead path with respect to $e$ and $\varphi$, then the following statements hold.
		\begin{itemize}
			\item[(a)] If $\min\{d_G(v_1), d_G(v_2)\} < \Delta$, then $V(K)$ is $\varphi$-elementary.
			\item[(b)] $|\overline{\varphi}(v_3) \cap (\overline{\varphi}(v_0) \cup \overline{\varphi}(v_1))| \leq 1$.
		\end{itemize}
	\end{lem}
	
	Let $B$ be a short broom with respect to $xy$ and $\varphi$, and let $v_{\ell_1}, \ldots, v_{\ell_k}$ be a subsequence of $v_1, \ldots, v_p$. This subsequence is called an \textit{$\alpha_0$-sequence} with respect to $\varphi$ and $B $ if it satisfies the following conditions:
	
	\[
	\varphi(zv_{\ell_1}) = \alpha_{0} \in \overline{\varphi}(x) \cup \overline{\varphi}(y),\ \varphi(zv_{\ell_2}) = \alpha \in \overline{\varphi}(v_{\ell_1}),\ \varphi(zv_{\ell_i}) \in \overline{\varphi}(v_{\ell_{i-1}})\ \text{for } i \in [2, k].
	\]
	A vertex belonging to an $\alpha_0$-sequence is referred to as an \textit{$\alpha_0$-inducing vertex} with respect to $\varphi$ and $B$, and a missing color at such a vertex is called an \textit{$\alpha_0$-inducing color} (simply $\alpha_0$-inducing). For convenience, the color $\alpha_0$ itself is also considered an $\alpha_0$-inducing color (simply $\alpha_0$-inducing).
	
	%
	%
	Recently, Cao, Chen, Jing, Stiebitz and Toft \cite{Cao2019survey} proved the following result.
	
	\begin{lem}[\cite{Cao2019survey}]\label{lem-short broom}
		Let $G$ be a graph with maximum degree $\Delta$ and $\chi'(G) = \Delta + 1$, let $e \in E_G(x,y)$ be a critical edge, let $\varphi \in \mathcal{C}^\Delta(G - e)$ be a coloring, and let $B = (x, e,$ $y, e', z,$ $e_1,$ $v_1,\ldots, e_p, v_p)$ be a $\varphi$-broom at $(x,y)$ such that $\min\{d_G(y),d_G(z)\} < \Delta$. Then $V(B)$ is $\varphi$-elementary. Furthermore, if the $\varphi$-broom $B$ is a maximal $\varphi$-broom at $(x,y)$ such that $\min\{d(y),d(z)\} < \Delta$, then
		\begin{equation*}
			\sum_{v \in V(B) \backslash \{z\}} \left( d_G(v) + \mu_G(z,v) - \Delta \right) \geq 2.
		\end{equation*}
	\end{lem}
	
	\subsection{Proof of Theorem \ref{ne-broom}}
	
	
	\broom*
%
	\begin{proof}
		If $\min\{d_G(y),d_G(z)\} < \Delta$, then the result follows directly from Lemma \ref{lem-short broom}. Otherwise, we have $d_G(y) = d_G(z) = \Delta$. Suppose to the contrary that $\sum_{\alpha \in [1, \Delta]} m_{\varphi, B}(\alpha)$ $\geq 2$. Let $M_{\varphi, B} = \{\alpha \in [1, \Delta] \mid m_{\varphi, B}(\alpha) \geq 1\}$, writing simply $M$ when $\varphi$ and $B$ are understood. Then, $\sum_{\alpha \in M} m_{\varphi, B}(\alpha) = \sum_{\alpha \in [1, \Delta]} m_{\varphi, B}(\alpha) \geq 2$. 
		
		Without loss of generality, we may assume that $\varphi(yz) = \zeta \in \overline{\varphi}(x)$ and $\overline{\varphi}(y) = 1$. Then $(y, yx, x, yz, z)$ forms a multifan; so $\{x, y, z\}$ is $\varphi$-elementary, $\overline{\varphi}(x) \cup \overline{\varphi}(y)$ and $\varphi(x) \cap \varphi(y)$ form a partition of $[1, \Delta]$ by Lemma \ref{multifan}.
		Thus, by the definition of $m_{\varphi, B}(\alpha)$, the assumption $\sum_{\alpha \in M} m_{\varphi, B}(\alpha) \geq 2$ and the fact that $\{x, y, z\}$ is $\varphi$-elementary, we conclude that there exist colors $\alpha, \beta \in M$ and vertices $u, v \in V(B) \backslash \{x, y, z\}$ such that $\beta \in \overline{\varphi}(u)$ and $\alpha \in \overline{\varphi}(v)$. Note that $\alpha$ and $\beta$ are not necessarily distinct, and $u$ and $v$ are not necessarily distinct. However, the situation where $\alpha = \beta$ and $u = v$ simultaneously is impossible, as it would imply that the same color is missing twice at the same vertex.  
		
		We may assume that $\alpha$ is $\alpha_{0}$-inducing, $\beta$ is $\beta_{0}$-inducing, $\alpha_0 = \varphi(zv_1) \in \overline{\varphi}(x) \cup \overline{\varphi}(y)$, $\beta_0 = \varphi(zu_1) \in \overline{\varphi}(x) \cup \overline{\varphi}(y)$, $\alpha_{1} \in \overline{\varphi}(v_1)$ and $\beta_{1} \in \overline{\varphi}(u_1)$, where $v_1, u_1\in V(B)$.
		Without loss of generality, assume $v = v_t$, $u = u_s$, with integers $t, s \geq 1$, $\varphi(zv_i) = \alpha_{i-1}$, $\alpha_{i} \in \overline{\varphi}(v_i)$ for $i=1, 2, \ldots, t$, $\varphi(zu_j) = \beta_{j-1}$, $\beta_{j} \in \overline{\varphi}(u_j)$ for $j=1, 2, \ldots, s$,  $\alpha \in \overline{\varphi}(v_t) \cap M$ and $\beta \in \overline{\varphi}(u_s) \cap M$. Note that these two sequences may share common vertices even when $u \neq v$. When $u = v$, these two sequences are identical, meaning that $s = t$ and $u_i = v_i$ for all $i = 1, \ldots, t$. It is easy to check that there is a short broom $B'$ with $V(B') =$ $\{x, y, z, v_1, v_2, \ldots, v_t\} \cup \{u_1, u_2, \ldots, u_s\}$ and $E(B') = \{xy, yz , zv_1,\ldots, zv_t\}$ $\cup$ $\{zu_1, zu_2,\ldots, zu_s\}$.
		In the following proof, we choose $v=v_t$ and $u=u_s$ satisfying the following conditions:
		\newcounter{listcounter}
		\begin{itemize}
			\item[(a)] $\overline{\varphi}(v) \cap M \neq \emptyset$, $\overline{\varphi}(u) \cap M \neq \emptyset$;
			\item[(b)]\refstepcounter{listcounter}\label{item-b} subject to (a), $t$ and $s$ are both as small as possible.\ \ \  $(\ast)$
		\end{itemize}	
		By the minimality of $t$ and $s$, $\alpha \notin \{\alpha_{1}, \alpha_{2}\ldots, \alpha_{t-1}\}$ (respectively, $\beta \notin  \{\beta_{1},\beta_2, \ldots, \beta_{s-1}\}$). Moreover, if $v_i = u_j$ for some $i$ and $j$, then $i = j$; otherwise, it would contradict the minimality of $t$ and $s$.
		
		\begin{claim}\label{a-1}
			The following statements hold:
			\begin{itemize}
				\item[(a)] If $\alpha \in \overline{\varphi}(v_t) \cap M \cap (\overline{\varphi}(x) \cup \overline{\varphi}(y))$, then $\alpha \neq \alpha_0$ and $zv_1 \notin P_x(\alpha, \alpha_0)$.
				\item[(b)] If $\beta \in \overline{\varphi}(u_s) \cap M \cap (\overline{\varphi}(x) \cup \overline{\varphi}(y))$, then $\beta \neq \beta_0$ and $zu_1 \notin P_x(\beta, \beta_0)$.
			\end{itemize}
		\end{claim}
		\begin{proof}
			We first prove (a). Assume that $\alpha \in \overline{\varphi}(v_t) \cap M \cap (\overline{\varphi}(x) \cup \overline{\varphi}(y))$. 
			To show $\alpha \neq \alpha_0$, suppose to the contrary that $\alpha = \alpha_0$. Then $t\geq 2$, and we can conclude that $\alpha = \alpha_0 \in \overline{\varphi}(x)$. Otherwise, $\alpha = \alpha_0 = 1 \in \overline{\varphi}(y)$. 
			Coloring $xy$ by $\zeta$, recoloring $yz$ by 1, and uncoloring $zv_1$ result in an edge coloring in $\mathcal{C}^\Delta(G-zv_1)$, say $\varphi_1$, and $(z, zv_1, v_1, zv_2, \ldots, zv_t)$ is a multifan with respect to $\varphi_1$. However, this contradicts Lemma \ref{multifan} because $1 \in \overline{\varphi_1}(v_1) \cap \overline{\varphi_1}(v_t)$. 
			Hence,  $\alpha = \alpha_0 \in \overline{\varphi}(x)$.
			Next we have $zv_1 \notin P_x(1, \alpha_0)$, because otherwise, we could recolor $zv_t$ by $\alpha_0$, recolor $zv_{t-1}$ by $\alpha_{t-1}$,$\ldots$, recolor $zv_1$ by $\alpha_{1}$ to obtain a coloring $\varphi_2 \in \mathcal{C}^{\Delta}(G-xy)$, where $\alpha_0 \in \overline{\varphi_2}(v_1)$, contradicting the fact that $x$ and $y$ are $(1, \alpha_0)$-linked with respect to $\varphi_2$ by Lemma \ref{multifan}. 
			Thus, $zv_1 \notin P_x(1, \alpha_0)$. We can then do $(1, \alpha_0)$-swap(s) at both $z$ and $v_{t}$, which brings us back to the previous case in which $\alpha = \alpha_0 = 1$. Therefore, $\alpha \neq \alpha_0$.
			
			Now, to show $zv_1 \notin P_x(\alpha, \alpha_0)$, suppose otherwise that $zv_1 \in P_x(\alpha, \alpha_0)$. Recall that $\alpha, \alpha_0 \in \overline{\varphi}(x) \cup \overline{\varphi}(y)$. 
			Then $t \geq 2$ and there is exactly one of $\alpha$ and $\alpha_{0}$ is missed at $x$; otherwise, $zv_1 \notin P_x(\alpha, \alpha_0)$. So $x$ and $y$ is $(\alpha, \alpha_{0})$-linked by Lemma \ref{multifan}. 
			A $(\alpha, \alpha_0)$-swap at $v_t$ results in $\alpha = \alpha_{0}$ without changing the colors of edges which are incident with $x$ or $y$ (This fact will be used very often without mentioning.), contradicting the fact that $\alpha\neq \alpha_0$.
			
			The proof of (b) follows by the symmetry of $u$ and $v$, interchanging the roles of $u_1$ and $v_1$, $\beta$ and $\alpha$, $\beta_0$ and $\alpha_0$. 
		\end{proof}
		
		\begin{claim}\label{a-2}
			If $\alpha \in \overline{\varphi}(v_t) \cap M \cap (\overline{\varphi}(x) \cup \overline{\varphi}(y))$, then we may assume $\varphi(zv_1) = 1$.
		\end{claim}
		\begin{proof}
			Assume $\alpha_0 \neq 1$; then $\alpha_0 \in \overline{\varphi}(x)$. And $\alpha \neq \alpha_0$ by Claim \ref{a-1}.
			If $\alpha \neq 1$, then $\alpha \in \overline{\varphi}(x)$, so $x$ and $y$ are $(1, \alpha)$-linked with respect to $\varphi$ by Lemma \ref{multifan}; and a $(1, \alpha)$-swap at $v_t$ results in $1 \in \overline{\varphi}(v_t)$ without changing the colors of edges which are incident with $x$ or $y$. 
			Thus, we may assume $1 \in \overline{\varphi}(v_t)$ in any case. By Claim \ref{a-1}, we have $zv_1 \notin P_x(1, \alpha_0)$.
			A $(1, \alpha_0)$-swap at $z$ then gives $\varphi(zv_1) = 1 = \overline{\varphi}(y)$, so we may assume $\varphi(zv_1) = 1$. Note that these operations may change some colors of $zu_j$ for $j=1, 2,\ldots, s$, but we are still in this case as $1, \alpha, \alpha_0 \in \overline{\varphi}(x) \cup \overline{\varphi}(y)$. 
		\end{proof}
		
		Recall that $\overline{\varphi}(x) \cup \overline{\varphi}(y)$ and $\varphi(x) \cap \varphi(y)$ form a partition of $[1, \Delta]$. We therefore consider the following cases.
		
		\textbf{Case A:} $\sum_{\alpha \in M \cap (\overline{\varphi}(x) \cup \overline{\varphi}(y))} m_{\varphi, B}(\alpha) \geq 2$.

		In this case, we utilize the sequences and short broom structure defined previously. Without loss of generality, we may assume that $\alpha \in \overline{\varphi}(v) \cap M \cap (\overline{\varphi}(x) \cup \overline{\varphi}(y))$ and  $\beta \in \overline{\varphi}(u) \cap M \cap (\overline{\varphi}(x) \cup \overline{\varphi}(y))$. By Claim \ref{a-2}, we can also assume $\varphi(zv_1) = 1$. 
		
		We now claim that $v_1 \neq u_1$. Suppose, for contradiction, that $v_1 = u_1$.
		Coloring $xy$ by $\zeta$, recoloring $yz$ by 1 and uncoloring $zv_1$ give an edge $\Delta$-coloring $\varphi_1 \in \mathcal{C}^\Delta(G-zv_1)$. And there is a multifan with center $z$ and vertex set $\{v_1, v_2, \ldots, v_t\} \cup \{u_1, u_2, \ldots, u_s\}$ under $\varphi_1$. 
		By Lemma \ref{multifan}, under $\varphi_1$, the colors $1, \zeta, \beta$ and $\alpha$ are pairwise distinct; $\alpha, \beta\in\overline{\varphi_1}(x)$; $z$ and $v_1$ are $(1, \zeta)$-linked; $z$ and $v_t$ are $(\alpha, \zeta)$-linked; and $z$ and $u_s$ are $(\beta, \zeta)$-linked. 
		Regardless of whether $u_s = v_t$ or not, doing $(\zeta, \alpha)-(1, \zeta)-(\zeta, \beta)-(\zeta, \alpha)$-swaps at $x$ yields $P_z(1, \zeta) = (z, zy, y, yx, x)$, a contradiction. Hence, $v_1 \neq u_1$.
		
		Recall that $\varphi(zv_i) = \alpha_{i-1}$, $\alpha_{i} \in \overline{\varphi}(v_i)$ for $i=1, 2, \ldots, t$, $\varphi(zu_j) = \beta_{j-1}$, $\beta_{j} \in \overline{\varphi}(u_j)$ for $j=1, 2, \ldots, s$, $\alpha_0=1$, $\alpha \in \overline{\varphi}(v_t)\cap (\overline{\varphi}(x) \cup \overline{\varphi}(y))$ and $\beta \in \overline{\varphi}(u_s)\cap (\overline{\varphi}(x) \cup \overline{\varphi}(y))$. The sets $\{v_1, v_2,\ldots, v_t\}$ and $\{u_1, u_2,\ldots, u_s\}$ are disjoint, and colors $\alpha_{1},\ldots,\alpha_{t}$, $\beta_{1},\ldots,\beta_{s}$ are pairwise distinct; otherwise, we return to the previous case in which $v_1 = u_1$ by the minimality of $t$ and $s$. We now consider the following subcases. 
		
		\textbf{Subcase A.1:} $\alpha = \beta$.
		
		We first claim that $yz \in P_{v_t}(\beta, \zeta)$ and $P_{v_t}(\beta, \zeta)$ meets $z$ before $y$  under $\varphi$. Coloring $xy$ by $\zeta$, recoloring $yz$ by 1, and uncoloring $zv_1$ give an edge $\Delta$-coloring $\varphi_1 \in \mathcal{C}^\Delta(G-zv_1)$, and form $(z, zv_1, v_1, zv_2, v_2,\ldots, zv_t, v_t)$ a multifan under $\varphi_1$; so $1 \notin \{\beta, \beta_0\}$, and $z$ is $(\beta, \zeta)$-linked with $v_t$ under $\varphi_1$ by Lemma \ref{multifan}. Thus, $yz \in P_{v_t}(\beta, \zeta)$ and $P_{v_t}(\beta, \zeta)$ meets $z$ before $y$  under $\varphi$.
		
		Denote by $P_z^*(1, \beta_0)$ the maximal subpath of the chain $P_z(1, \beta_0)$ that starts at $z$ along $zv_1$ and does not include the edge $zu_1$. If $P_z^*(1, \beta_0)$ does not end at $y$, a sequence of Kempe changes (similar to the proof of Lemma 5 in \cite{CCSVS2022}) yields an edge $\Delta$-coloring of $G$, a contradiction. 
		The Kempe changes are: 
		\[
		\begin{bmatrix}
			zu_s & zu_{s-1} &\ldots & zu_{1} & P_{[v_t, z]}(\beta, \zeta) & yz & P_z^*(1, \beta_0) & xy \\
			\beta_{s-1} \to \beta & \beta_{s-2} \to \beta_{s-1} &\ldots & \beta_{0} \to \beta_1 & \beta / \zeta & \zeta \to 1 & 1 / \beta_0 & \zeta
		\end{bmatrix}.
		\]
		These cases therefore lead to a contradiction, and we may assume henceforth that $P_z^*(1, \beta_0)$ ends at $y$.
		
		Recall that $\beta_0, \beta$, $\zeta \in \overline{\varphi}(x)$ and $\overline{\varphi}(y) = 1$. Then, uncolor $zv_1$, do a $(1, \beta_0)$-swap on $P_{[v_1, y]}(1, \beta_0)$ and color $xy$ by $\beta_0$. It is easy to check that the resulting coloring $\varphi_1$ is an edge $\Delta$-coloring of $G-zv_1$. 
		Now, observe that $\beta_0 \in \overline{\varphi_1}(v_1)$ and $\varphi_1(zu_1) = \beta_0$, which implies that $(z, zv_1, v_1, zv_2, v_2,\ldots, v_t,$ $zu_1, u_1,\ldots,$ $zu_s, u_s)$ is a multifan. However, we also have $\beta \in \overline{\varphi_1}(u_s) \cap \overline{\varphi_1}(v_t)$, contradicting Lemma \ref{multifan} $(a)$.
		
		\textbf{Subcase A.2:} $\alpha \neq \beta$.
		
		In this case, by Claim \ref{a-1}, we have $\alpha \neq 1$ and $\beta \neq \beta_0$. If $\alpha = \zeta$, a $(1, \zeta)$-swap at $v_t$ together with Claim \ref{a-1} give a contradiction. Thus, $\alpha \neq \zeta$.
		We next claim that $\alpha \neq \beta_0$. Otherwise, we can get a contradiction by the same operations as $v_1 = u_1$ with $\beta_0$ playing the role of $\beta$ and $\beta$ playing the role of $\alpha$. Thus, $\alpha \notin \{1, \zeta, \beta, \beta_0\}$ and $\beta \neq \beta_0$. We then consider the following subcases based on whether $\beta \in \{1, \zeta\}$. 
		
		Suppose $\beta \in \{1, \zeta\}$. Note that the case $\beta = \zeta$ can be reduced to the case $\beta = 1$ by doing a $(1, \zeta)$-swap at $u_s$. Hence, we only consider $\beta = 1$. 
		By Claim \ref{a-1}, we have $zu_1 \notin P_x(1, \beta_0)$.
		We then do $(1, \beta_0)$-swap(s) at both $z$ and $u_s$, which brings us back to the previous cases in which $\alpha = \beta_0$ by the symmetry of $u$ and $v$.
		
		Suppose $1, \alpha, \beta, \zeta$ and $\beta_0$ are pairwise distinct. 
		By Claim \ref{a-1}, we have $zv_1 \notin P_x(1, \alpha)$. 
		We do the following Kempe changes in sequence: a $(1, \alpha)$-swap at $x$, a $(\zeta, \alpha)$-swap at $v_t$, and a $(\beta, \alpha)$-swap at $u_s$. Let $\varphi_2$ denote the resulting edge $\Delta$-coloring of $G-xy$.
		We then claim that $zu_1 \notin P_x(\beta_0, \alpha, \varphi_2)$. Otherwise, doing a $(\beta_0, \alpha)$-swap at $u_s$, then recoloring $zu_s$ by $\beta_0$, $zu_{s-1}$ by $\beta_{s-1}$,$\ldots$, $zu_1$ by $\beta_{1}$ would yield an edge $\Delta$-coloring $\varphi_3 \in \mathcal{C}^{\Delta}(G-xy)$ contradicting the fact that $x$ and $y$ are $(\beta_0, \alpha)$-linked under $\varphi_3$ by Lemma \ref{multifan}. Hence, $zu_1 \notin P_x(\beta_0, \alpha, \varphi_2)$.  
		Next, do a $(\beta_0, \alpha)$-swap at $x$ and a $(\zeta, \beta_0)$-swap at $v_t$.
		If $z \notin P_x(1, \beta_0)$, then by a $(1, \beta_0)$-swap at $x$, we are back to the previous case in which $\alpha = \beta_0$. Hence, $z \in P_x(1, \beta_0)$.
		Denote by $P_{z}^*(1, \beta_0)$ the maximal subpath of the chain $P_{z}(1, \beta_0)$ that starts at $z$ along $zu_1$ and does not include the edge $zv_1$.  
		Since $z \in P_x(1, \beta_0)$, $P_{z}^*(1, \beta_0)$ ends at $x$ or $y$, we do the following operations to get an edge $\Delta$-coloring of $G$:
		\[
		\begin{bmatrix}
			P_z^*(1, \beta_0) & zv_1 & zv2 &\ldots & zv_{t-1} & zv_t & xy \\
			1 / \beta_0 &1 \to \alpha_1 & \alpha_1 \to \alpha_2 &\ldots & \alpha_{t-2} \to \alpha_{t-1}& \alpha_{t-1} \to \beta_{0} & *
		\end{bmatrix},
		\]
		where the color $*$ for $xy$ is $\beta_0$ if $P_{z}^*(1, \beta_0)$ ends at $x$, and $1$ if it ends at $y$.
		This gives a contradiction to the assumption that G is $\Delta$-critical.
		
		\textbf{Case B:} $\sum_{\alpha \in M \cap (\overline{\varphi}(x) \cup \overline{\varphi}(y))} m_{\varphi, B}(\alpha) = 1$.
		
		In this case, there exists exactly one color $\alpha \in M \cap (\overline{\varphi}(x) \cup \overline{\varphi}(y))$ with $m_{\varphi,B}(\alpha) = 1$. Let $v$ be the unique vertex in $V(B) \backslash \{x, y, z\}$ such that $\alpha \in \overline{\varphi}(v) \cap M \cap (\overline{\varphi}(x) \cup \overline{\varphi}(y))$. Without loss of generality, we may take $v =v_t$ for some $t \geq 1$ with $\varphi(zv) = \alpha_{t-1}$ being $\alpha_0$-inducing. By Claim 2, we may assume $\varphi(zv_1) = 1$. 
		
		Since $\sum_{\alpha \in [1, \Delta]} m_{\varphi,B}(\alpha) \geq 2$, there exists some $\beta \in \varphi(x) \cap \varphi(y)$ and at least two distinct vertices $w, w' \in V(B) \backslash \{x, y, z\}$ such that $\beta \in \overline{\varphi}(w) \cap \overline{\varphi}(w')$. Then we may take $w' = u_s$ for some $s \geq 1$ with $\varphi(zw') = \beta_{s-1}$ being $\beta_0$-inducing. 
		We may assume $\varphi(zw) = \gamma_{l-1}$ is $\gamma_0$-inducing with integer $l \geq 1$, $\gamma_0 = \varphi(zw_1) \in \overline{\varphi}(x) \cup \overline{\varphi}(y)$, where $w_1\in V(B)$. 
		Without loss of generality, assume $w = w_l$, $\varphi(zw_i) = \gamma_{i-1}$, $\gamma_{i} \in \overline{\varphi}(w_i)$ for $i=1, 2, \ldots, l$ and $\beta \in \overline{\varphi}(w_l) \cap \varphi(x) \cap \varphi(y)$.
		Note that these three sequences may share common vertices.
		We claim that $u_1$ and $w_1$ cannot both be equal to $v_1$. Otherwise, suppose to the contrary that $u_1 = w_1 = v_1$. Coloring $xy$ by $\zeta$, recoloring $yz$ by $1$ and uncoloring $zv_1$ yield a new edge $\Delta$-coloring, say $\varphi_1$, and there is a multifan with center $z$ and vertex set $\{v_1, \ldots, v_t\} \cup \{u_1, \ldots, u_s\} \cup \{w_1, \ldots, w_{l}\}$. However, $\beta \in \overline{\varphi_1}(w) \cap \overline{\varphi_1}(w')$, which contradicts Lemma \ref{multifan} $(a)$. We then consider the following subcases.
		
		\textbf{Subcase B.1:} One of $u_1$ and $w_1$ is equal to $v_1$.
		
		Without loss of generality, assume that $w_1 = v_1$.
		We first claim that $x$ is $(\beta, \beta_{0})$-linked with $w$. Otherwise, suppose that $x$ is not $(\beta, \beta_{0})$-linked with $w$. 
		Then doing a $(\beta, \beta_{0})$-swap at $w$ gives an edge $\Delta$-coloring of $G-xy$, say $\varphi_2$. If this Kempe change does not alter the color of $zu_1$ (i.e. $\varphi_2(zu_1) = \varphi (zu_1) = \beta_{0}$), then we are back to Case A by taking $\varphi=\varphi_2$ since $\alpha \in \overline{\varphi_2}(v) \cap (\overline{\varphi_2}(x) \cup \overline{\varphi_2}(y))$, $\beta_0 \in \overline{\varphi_2}(w) \cap (\overline{\varphi_2}(x) \cup \overline{\varphi_2}(y))$ and there still is a broom with vertex set $\{x, y, z\} \cup \{v_1, \ldots, v_t\} \cup \{u_1, \ldots, u_s\} \cup \{w_1, \ldots, w_{l}\}$ under $\varphi_2$. Hence, $\varphi_2(zu_1) =\beta$. 
		However, there still is a broom with vertex set $\{x, y, z\} \cup \{v_1, \ldots, v_t\} \cup \{w_1, \ldots, w_{l}\}$ under $\varphi_2$ due to $\beta_{0} \in \overline{\varphi_2}(x)$. So we are also back to Case A by taking $\varphi=\varphi_2$. 
		Thus, $x$ is $(\beta, \beta_{0})$-linked with $w$.
		
		Then doing a $(\beta, \beta_{0})$-swap at $w'$ yields a coloring  $\varphi_3$ where $\beta_0 \in \overline{\varphi_3}(w')$. 
		If this Kempe change does not alter the color of $zu_1$, i.e., $\varphi_3(zu_1) = \varphi(zu_1) = \beta_{0}$, then $(x, xy,$ $y,$ $yz, z, zv_1, v_1,$ $\ldots, zv_t, v_t,$ $zu_1, u_1, \ldots,$ $zu_s, u_s)$ remains a short broom under $\varphi_3$. 
		If this Kempe change does alter the color of $zu_1$, i.e., $\varphi_3(zu_1) = \beta$, then there still is a broom with vertex set $\{x, y, z\} \cup \{v_1, \ldots, v_t\} \cup \{u_1, \ldots, u_s\} \cup \{w_1, \ldots, w_{l}\}$ under $\varphi_3$ because $\beta \in \overline{\varphi_3}(w)$ and $\beta_0 \in \overline{\varphi_3}(x)$. 
		Thus, we are back to Case A by taking $\varphi = \varphi_3$.
		
		\textbf{Subcase B.2:} Neither $u_1$ nor $w_1$ is equal to $v_1$.
		
		We claim that $w$ is not $(1, \beta)$-linked with $w'$. Otherwise, $w$ is $(1, \beta)$-linked with $w'$. Then doing a $(1, \beta)$-swap at $w'$ gives an edge $\Delta$-coloring of $G-xy$, say $\varphi_4$. 
		If this Kempe change does not alter the color of $zv_1$, i.e., $\varphi_4(zv_1) = \varphi(zv_1) = 1$, then there is a short broom with vertex set $\{x, y, z\} \cup \{v_1, \ldots, v_t\} \cup \{u_1, \ldots, u_s\} \cup \{w_1, \ldots, w_{l}\}$ under $\varphi_4$ and $1 \in \overline{\varphi_4}(w) \cap \overline{\varphi_4}(w')$, and we are back to Case A by taking $\varphi = \varphi_4$. 
		If this Kempe change does alter the color of $zv_1$, i.e., $\varphi_4(zv_1) = \beta$, then there is a short broom with vertex set $\{x, y, z\} \cup \{u_1, \ldots, u_s\} \cup \{w_1, \ldots, w_{l}\}$ under $\varphi_4$, and we are also back to Case A by taking $\varphi = \varphi_4$. Thus, $w$ is not $(1, \beta)$-linked with $w'$. 
		
		As at least one of $w$ and $w'$ is not $(1, \beta)$-linked with $y$, without loss of generality, we assume that $w'$ is not $(1, \beta)$-linked with $y$. Then doing a $(1, \beta)$-swap at $w'$ gives an edge $\Delta$-coloring of $G-xy$, say $\varphi_5$. Regardless of whether this Kempe change alter the color of $zv_1$ or not, there is still a short broom with vertex set $\{x, y, z\} \cup \{v_1, \ldots, v_t\} \cup \{u_1, \ldots, u_s\} \cup \{w_1, \ldots, w_{l}\}$ because $\beta \in \varphi_5(w)$ and $1 \in \varphi_5(w')$. And then, we are back to Case A by taking $\varphi = \varphi_5$.

		\textbf{Case C:} $\sum_{\alpha \in M \cap (\overline{\varphi}(x) \cup \overline{\varphi}(y))} m_{\varphi, B}(\alpha) = 0$.
		
		Since $\sum_{\alpha \in [1, \Delta]} m_{\varphi, B}(\alpha) \geq 2$, there exists some $\beta \in \varphi(x) \cap \varphi(y)$ and at least two distinct vertices $u, v \in V(B) \backslash \{x, y, z\}$ such that $\beta \in \overline{\varphi}(u) \cap \overline{\varphi}(v)$. Then we may take $u = u_s$ for some $s \geq 1$ with $\varphi(zu) = \beta_{s-1}$ being $\beta_0$-inducing, and $v = v_t$ for some $t \geq 1$ with $\varphi(zv) = \alpha_{t-1}$ being $\alpha_0$-inducing. 
		Similar to Case B, we have if $v_1 = u_1$, then $\varphi(zv_1) = \alpha_{0} \neq 1$. Moreover, for the cases $v_1 = u_1$, $v_1 \neq u_1$ and $1 \notin \{\alpha_{0}, 
		\beta_{0}\}$, we are back to Case A or B following the similar argument to Subcase B.2 with $u$ playing the role of $w$ and $v$ playing the role of $w'$.

		Now, suppose $v_1 \neq u_1$ and $1 \in \{\alpha_{0}, \beta_{0}\}$. Without loss of generality, we assume $\alpha_0 = 1$. Since $d_G(z) = \Delta$, both $1$ and $\beta$ are present at $z$. Let $\varphi(zz_\beta) = \beta$. By the minimality of $t$ and $s$ (Recall the condition $(\ast)$), we have $z_\beta \notin \{v_1, v_2, \ldots, v_t\} \cup \{u_1, u_2, \ldots, u_s\}$.
		
		We first claim that $u$ is $(1, \beta)$-linked with $x$. Otherwise, suppose $u$ is not $(1, \beta)$-lined with $x$. Then a $(1, \beta)$-swap at $u$ gives an edge $\Delta$-coloring of $G-xy$, say $\varphi_1$. 
		And we are back to Case B or A by taking $\varphi = \varphi_1$ since $1 \in \overline{\varphi_1}(u) \cap (\overline{\varphi_1}(x) \cup \overline{\varphi_1}(y))$ and regardless of whether this Kempe change alters the color of $zv_1$ or not,
		$(x, xy, y, yz, z, $$zu_1, u_1, \ldots,$ $zu_s, u_s)$ is still a short broom under $\varphi_1$.
		Thus, $u$ is $(1, \beta)$-lined with $x$. 
		Then a $(1, \beta)$-swap at $v$ gives an edge $\Delta$-coloring of $G-xy$, say $\varphi_2$. 
		Regardless of whether this Kempe change alters the color of $zv_1$ or not,
		$(x, xy, y, yz, z, $$zu_1,$ $u_1, \ldots,$ $zu_s, u_s, zv_1, v_1,$ $\ldots, zv_t, v_t,)$ is still a short broom under $\varphi_2$ due to $\beta \in \overline{\varphi_2}(u)$. So we are back to Case B or A by taking $\varphi = \varphi_2$.
	\end{proof}
	
	\section{Proof of Theorem \ref{vertex-splitting thm}}
	
	\begin{lem} \label{tkpl}
		Let $G$ be a $\Delta$-critical graph with an edge $xy\in E(G)$, and let $\varphi \in \mathcal{C}^{\Delta}(G-xy)$. Suppose $$K = (x, xy, y, yz, z, zu, u)\ \  \mbox{and}\ \  K^* = (x, xy, y, yz, z, zv, v)$$ are two Kierstead paths with respect to $xy$ and $\varphi$. If $\overline{\varphi}(u) \cup \overline{\varphi}(v) \subseteq \overline{\varphi}(x) \cup \overline{\varphi}(y)$, then $\max\{d(u), d(v)\} = \Delta$.
	\end{lem}
	
	\begin{proof}
		Since $K = (x, xy, y, yz, z, zu, u)\ \  \mbox{and}\ \  K^* = (x, xy, y, yz, z, zv, v)$ are two Kierstead paths with respect to $xy$ and $\varphi$, $(x, xy, y, yz, z, zu, u, zv, v)$ is a short broom. Then  the result follows directly from Theorem \ref{ne-broom}.
	\end{proof}
	
	Let $G$ be a graph and let $u, v \in V(G)$ be adjacent. We call $(u, v)$ a \emph{full-deficiency pair} of $G$ if $d(u) + d(v) = \Delta(G) + 2.$
	Note that if we split a vertex $x$ of a $\Delta$-regular graph into vertices $u$ and $v$, then we obtain a full-deficiency pair $(u, v)$. 
	If $G$ is $\Delta$-critical, then a full-deficiency pair $(u, v)$ of $G$ is called a \emph{saturating pair} of $G - uv$ in \cite{BVA2020}.
	
	\begin{lem}[\cite{CCSVS2022}]\label{fdpl}
		If $G$ is an $n$-vertex class 2 graph with a full-deficiency pair $(x, y)$ such that $xy$ is a critical edge of $G$, then $G$ satisfies the following properties.
		\begin{itemize}
			\item[(i)] For every $a \in (N_G(x) \cup N_G(y)) \backslash \{x, y\}$, $d_G(a)=\Delta$.
			\item[(ii)] For every $a \in V(G) \backslash \{x, y\}$, if $dist(a, \{x, y\})=2$, then $d_G(a) \geq \Delta-1$. If also $d(x)<\Delta$ and $d_G(y)<\Delta$, then $d_G(a)=\Delta$.
			\item[(iii)] For every $a \in V(G) \backslash \{x, y\}$, if $d_G(a) \geq n-|N_G(x) \cup N_G(y)|$, then $d_G(a) \geq \Delta-1$. If also $d_G(x)<\Delta$ and $d(y)<\Delta$, then $d(a)=\Delta$.
		\end{itemize}
	\end{lem}
	
	\begin{cor}\label{v-s c}
		Let $G$ be an $n$-vertex class 2 graph with a full-deficiency pair $(x, y)$, and let $xy$ be a critical edge of $G$. If $\Delta \geq \frac{2(n - 1)}{3}$, then there exists at most one vertex $u \in V(G) \backslash \{x, y\}$ such that $d_G(u)=\Delta-1$.
	\end{cor}
	
	\begin{proof}
		We proceed by contradiction. Suppose that there exist two distinct vertices $u, v \in V(G) \backslash \{x, y\}$ such that $d_G(u)=d_G(v)=\Delta-1$. According to Lemma \ref{fdpl} $(i)$, $u, v \notin (N_G(x) \cup N_G(y)) \backslash \{x, y\}$. Since $x$ and $y$ are adjacent in $G$ and $d_G(x)+d_G(y)=\Delta+2$, we have $|N_G(x) \cap N_G(y)| \leq \frac{1}{2}\Delta$, which implies that $d_G(u)+|N_G(x) \cup N_G(y)| \geq \Delta-1+(\frac{1}{2}\Delta+2)=\frac{3}{2}\Delta+1 \geq (n - 1) + 1 = n$. Hence, by Lemma \ref{fdpl} $(iii)$, we may assume that one of $x$ and $y$, w.l.o.g., $y$ has degree $\Delta$. Thus $d_G(x)=2$. Let $y'$ denote the other neighbor of $x$ in $G$. As $xy$ is a critical edge of $G$ and $d_G(x)=2$ and $d_G(y)=\Delta$, Lemma 2 (VAL) implies that $d_G(y')=\Delta$. Clearly, $y' \notin \{u, v\}$ and $(x, y')$ is a full-deficiency pair of $G$. Thus, similarly, we may assume $u, v \notin N_G(y')$ by Lemma \ref{fdpl} $(i)$.
		
		Since $d_G(u) = d_G(v) = \Delta-1$ and $x, y, y' \notin N_G(u) \cup N_G(v)$, we obtain $|N_G(u) \cap N_G(v)| \geq 2(\Delta - 1) - (n - 3) \geq \frac{n-1}{3}$. 
		Furthermore, since $u, v \notin N_G(y) \cup (N_G(u) \cap N_G(v))$, we have $|N_G(y) \cap N_G(u) \cap N_G(v)| \geq \Delta + \frac{n-1}{3}-(n-2)=1$. Let $z \in N_G(y) \cap N_G(u) \cap N_G(v)$. 
		As $\{x, y\}$ is $\varphi$-elementary, $|\overline{\varphi}(x) \cup \overline{\varphi}(y)|=|\overline{\varphi}(x)| + |\overline{\varphi}(y)| = \Delta$, and so $\overline{\varphi}(x) \cup \overline{\varphi}(y)=[1, \Delta]$. 
		Thus $K=(x, xy, y, yz, z, zu, u)$ and $K^*=(x, xy, y, yz, z, zv, v)$ are two Kierstead paths with respect to $xy$ and $\varphi$ and $\overline{\varphi}(u) \cup \overline{\varphi}(v) \subseteq \overline{\varphi}(x) \cup \overline{\varphi}(y)$. 
		However, $d_G(u) = d_G(v)=\Delta-1$, contradicting Lemma \ref{tkpl}.
	\end{proof}
	
	Using the proof technique from \cite{CCSVS2022} together with Corollary \ref{v-s c} of this paper, we obtain the following theorem.
	
	\splittingthm*
	
	\begin{proof}
		Since $G$ is obtained from $G^*$ by a vertex-splitting, then $G$ is Overfull, which implies that it is class 2. To establish that $G$ is $\Delta$-critical, it suffices to show that each edge of $G$ is critical. Suppose to the contrary that there is an edge $xy \in E(G)$ which is not critical. Let $G' =  G - xy$. Then $\chi'(G') = \Delta + 1$.
		
		Let $ab$ be the edge of $G$ whose contraction yields an $(n-1)$-vertex $\Delta$-regular class 1 graph. Then $(a, b)$ is a full-deficiency pair in $G$ with $ab$ being a critical edge, which implies $ab \neq xy$. Moreover, every vertex in $V(G) \backslash \{a, b\}$ has degree $\Delta$. Since any $\Delta$-coloring of $G - ab$ induces a $\Delta$-coloring of $G' - ab$, the edge $ab$ remains critical in $G'$. 
		If $\{a, b\} \cap \{x, y\} = \emptyset$, then $(a, b)$ is a full-deficiency pair in $G'$ with $d_{G'}(x) = d_{G'}(y) = \Delta - 1$, contradicting Corollary \ref{v-s c}. Therefore, $\{a, b\} \cap \{x, y\} \neq \emptyset$, and without loss of generality, let $a = x$. This gives $d_{G'}(a) + d_{G'}(b) = \Delta + 1$. However, Lemma \ref{VAL} (VAL) requires $d_{G'}(a) + d_{G'}(b) \geq \Delta + 2$, yielding a contradiction.
	\end{proof}
	
	\section{Proof of Theorem \ref{overfull thm}}
	
%
	\begin{lem}[\cite{caoovf2022}]\label{5vkp}
		Let $G$ be a $\Delta$-critical graph, $xy$ be a critical edge of $G$, $\varphi \in \mathcal{C}^{\Delta}(G-ab)$, and $K=(x, xy, y, yz, z, zs, s, st, t)$ be a Kierstead path with respect to $xy$ and $\varphi$. If $|\overline{\varphi}(t) \cap (\overline{\varphi}(x) \cup \overline{\varphi}(y))| \geq 3$, then the following holds.
		\begin{itemize}
			\item[(a)] There exists a coloring $\varphi^{*} \in \mathcal{C}^{\Delta}(G-xy)$ such that:
			\begin{itemize}
				\item[(i)] $\varphi^{*}(yz) \in \overline{\varphi}^{*}(x) \cap \overline{\varphi}^{*}(t)$;
				\item[(ii)] $\varphi^{*}(zs) \in \overline{\varphi}^{*}(y) \cap \overline{\varphi}^{*}(t)$;
				\item[(iii)] $\varphi^{*}(st) \in \overline{\varphi}^{*}(x)$;
			\end{itemize}
			\item[(b)] $d(y) = d(z) = \Delta$.
		\end{itemize}
	\end{lem}
	
	\begin{lem}\label{lfork}
		Let $G$ be a class 2 graph and $\varphi \in \mathcal{C}^{\Delta}(G - xy)$. Suppose $K=(x, xy, y, yz, z,$ $zs_1, s_1, s_1t_1, t_1)$ and $K^*=(x, xy, y, yz, z, zs_2, s_2, s_2t_2, t_2)$ are two Kierstead paths with respect to $xy$ and $\varphi$. If $\varphi(s_1t_1)=\varphi(s_2t_2)$, then $|\overline{\varphi}(t_1) \cap \overline{\varphi}(t_2) \cap (\overline{\varphi}(x) \cup \overline{\varphi}(y))| \leq 3$.
	\end{lem}
	
	\begin{proof}
		Let $\Gamma = \overline{\varphi}(t_1) \cap \overline{\varphi}(t_2) \cap (\overline{\varphi}(x) \cup \overline{\varphi}(y)) $. Suppose to the contrary that $ |\Gamma| \geq 4 $. By considering $ K $ and applying Lemma \ref{5vkp}, we conclude that $ d(y) = d(z) = \Delta $. We then follow the argument in \cite[Proof of Lemma 2.8]{caoovf2022} to show that there exists $ \varphi^* \in \mathcal{C}^\Delta(G - xy) $ such that:
		\begin{itemize}
			\item[(i)] $ \varphi^*(yz), \varphi^*(zs_2) \in \overline{\varphi}^*(x) \cap \overline{\varphi}^*(t_1) \cap \overline{\varphi}^*(t_2) $,
			\item[(ii)] $ \varphi^*(zs_1) \in \overline{\varphi}^*(y) \cap \overline{\varphi}^*(t_1) \cap \overline{\varphi}^*(t_2) $, and
			\item[(iii)] $ \varphi^*(s_1t_1) = \varphi^*(s_2t_2) \in \overline{\varphi}^*(x) $.
		\end{itemize}
		
		
		Let $ \alpha, \beta, \tau, \eta \in \Gamma $ be distinct, and let $ \varphi(s_1t_1) = \varphi(s_2t_2) = \gamma $. 
		We may assume that $ \alpha \in \overline{\varphi}(x) $ and $ \beta \in \overline{\varphi}(y) $. Otherwise, from $ d(y) = \Delta $, we either have $ \alpha \in \overline{\varphi}(y) $ and $ \beta \in \overline{\varphi}(x) $, or $ \alpha, \beta \in \overline{\varphi}(x) $. 
		If $ \alpha \in \overline{\varphi}(y) $ and $ \beta \in \overline{\varphi}(x) $, we may simply relabel $\alpha$ as $\beta$ and $\beta$ as $\alpha$. Thus, we only consider that $\alpha, \beta \in \overline{\varphi}(x)$. Let $\lambda \in \overline{\varphi}(y) $. Since $ x $ and $ y $ are $ (\beta, \lambda) $-linked, we do a $ (\beta, \lambda) $-swap at $ y $. This Kempe change may alter the colors on some edges of $ K $ and $ K^* $, but they are still Kierstead paths with respect to $ xy $ and the current coloring. Moreover, as $ d(y) = d(z) = \Delta $ and $K$ is a Kierstead path, we have $\overline{\varphi}(y) = \beta$ and $\gamma, \alpha, \tau, \eta \in \overline{\varphi}(x)$.
		
		Next, we may assume that $ \varphi(yz) = \alpha \in \overline{\varphi}(x) \cap \overline{\varphi}(t_1) \cap \overline{\varphi}(t_2)$. If not, since $K$ is a Kierstead path, we have $\varphi(yz) = \alpha' \in \overline{\varphi}(x)$. Since $ x $ and $ y $ are $ (\alpha', \beta) $-linked, we do an $ (\alpha', \beta) $-swap at both $ t_1 $ and $t_2$. Now we do an $ (\alpha, \beta) $-swap at both $t_1$ and $t_2$ as $ x $ and $ y $ are $ (\alpha, \beta) $-linked. Finally, we relabel $\alpha$ as $\alpha'$ and $\alpha'$ as $\alpha$. So now we get the color on $ yz $ to be $ \alpha \in \overline{\varphi}(x) \cap \overline{\varphi}(t_1) \cap \overline{\varphi}(t_2) $.
		
		%
		Then, we may assume that $\varphi(zs_1) = \beta \in \overline{\varphi}(y) \cap \overline{\varphi}(t_1) \cap \overline{\varphi}(t_2)$. Suppose to the contrary that $\varphi(zs_1)  = \beta' \neq \beta $. 
		Then $ \beta' \in \overline{\varphi}(x) $ since $ \overline{\varphi}(y) =  \beta  $ and $d_G(z) = \Delta$. If $ \beta' \in \Gamma $, we do $ (\beta, \gamma)-(\gamma, \beta') $-swaps at $ y $. 
		Thus, we assume $ \beta' \notin \Gamma $. 
		Furthermore, we can conclude that $ z \notin P_x(\beta, \beta') = P_y(\beta, \beta') $. 
		If not, we suppose $ z \in P_x(\beta, \beta') = P_y(\beta, \beta') $. We do a $ (\beta, \beta') $-swap at both $ t_1 $ and $ t_2 $. Since $ x $ and $ y $ are $ (\tau, \beta) $-linked, we do a $ (\beta, \tau) $-swap at both $t_1$ and $t_2$. Now we are back to the case when $\beta' \in \Gamma$. 
		Thus, $z \notin P_x(\beta, \beta') = P_y(\beta, \beta') $. 
		We do $ (\beta, \beta') $-swaps at $z$, $t_1$ and $t_2$. Since $x$ and $y$ are $ (\alpha, \beta) $-linked, we do $(\alpha, \beta)$-swap(s) at both $t_1$ and $t_2$. Then we do $(\beta, \gamma)-(\gamma, \tau)$-swaps at $y$. Now $ x $ and $ y $ are $ (\alpha, \tau)$-linked, we do $(\alpha, \tau)$-swap(s) at both $t_1$ and $t_2$. Then $(\tau, \eta)-(\eta, \gamma)-(\gamma, \beta)$-swaps at y give $\varphi(zs_1) = \beta \in \overline{\varphi}(y) \cap \overline{\varphi}(t_1) \cap \overline{\varphi}(t_2)$.
		
		Next, we may assume that $ \varphi(zs_2) = \eta \in \overline{\varphi}(x)\cap \overline{\varphi}(t_1) \cap \overline{\varphi}(t_2)$. Suppose otherwise that $ \varphi(zs_2) = \eta' $. Since $K^*$ is a Kierstead path with $\overline{\varphi}(y) = \beta$ and $d_G(z) = \Delta$, it follows that $ \eta' \in \overline{\varphi}(x)$.
		We first do a $(\beta, \eta')$-swap at both $t_1$ and $t_2$. Note that this Kempe change may alter the colors of $zs_1$ and $zs_2$. By the symmetry of $s_1$ and $s_2$, we may assume without loss of generality that it does not.
		We then do an $(\alpha, \beta)$-swaps at both $ t_1 $ and $ t_2 $ since $x$ and $y$ are $(\alpha, \beta)$-linked. Now $ z \in P_{t_1}(\beta, \gamma) $, we do a $ (\beta, \gamma) $-swap at $y$, then a $(\gamma, \tau)$-swap at $y$. Since $ x $ and $ y $ are $ (\tau, \alpha) $-linked, we do a $ (\tau, \alpha) $-swap at both $ t_1 $ and $ t_2 $. Finally, $(\tau, \eta)-(\eta, \gamma)-(\gamma, \beta)$-swaps at $y$ give a desired coloring.


		We have thus obtained the coloring $\varphi^*$ as desired. For convenience, we now rename $\varphi^*$ to $\varphi$. Thus, $\varphi$ satisfies conditions: (i) $ \varphi(yz), \varphi(zs_2) \in \overline{\varphi}(x) \cap \overline{\varphi}(t_1) \cap \overline{\varphi}(t_2) $, (ii) $ \varphi(zs_1) \in \overline{\varphi}(y) \cap \overline{\varphi}(t_1) \cap \overline{\varphi}(t_2) $ and (iii) $ \varphi(s_1t_1) = \varphi(s_2t_2) \in \overline{\varphi}(x) $.
		
		We first show that $ yz \in P_{t_1}(\alpha, \gamma) $. If not, doing an $ (\alpha, \gamma) $-swap at $ t_1 $ yields $\varphi_1 $, where $ P_y(\alpha, \beta, \varphi_1) = (y, yz, z, zs_1, s_1, s_1t_1, t_1)$. This contradicts the fact that $ x $ and $ y $ are $ (\alpha, \beta) $-linked with respect to $ \varphi_1 $. Next, we claim that $ P_{t_1}(\alpha, \gamma) $ meets $ z $ before $ y $. Otherwise, doing the following operations:
		\[
		\begin{bmatrix}
			s_1t_1 & P_{[s_1, y]}(\alpha, \gamma) & zs_1 & yz & xy \\
			\gamma \to \beta & \alpha / \gamma & \beta \to \alpha & \alpha \to \beta & \gamma
		\end{bmatrix}
		\]
		gives a $ \Delta $-coloring of $ G $, contradicting the assumption that $ G $ is $ \Delta $-critical. 
		Thus, we have that $ P_{t_1}(\alpha, \gamma) $ meets $z$ before $y$.
		Let $ P_z^*(\beta, \eta) $ be the $ (\beta, \eta) $-chain starting at $ z $ not including the edge $ zs_2 $. 
		Let $ P_z^*(\alpha, \gamma) $ be the $ (\alpha, \gamma) $-chain starting at $ z $ not including the edge $ yz $, which ends at $ t_1 $. 
		If $ P_z^*(\beta, \eta) $ not ends at $y$, we do the following operations to get a $ \Delta $-coloring of $ G $:
		\[
		\begin{bmatrix}
			P_z^*(\alpha, \gamma) & yz & P_z^*(\beta, \eta) & zs_2 & s_2t_2 & xy \\
			\alpha / \gamma & \alpha \to \beta & \beta / \eta & \eta \to \gamma & \gamma \to \eta & \alpha
		\end{bmatrix}.
		\]
		If $ P_z^*(\beta, \eta) $ ends at $y$, then uncoloring $us_1$, doing a $(\beta, \eta)$-swap on $P_{[s_1, y]}(\beta, \eta)$ and coloring $xy$ by $\eta$ give a $\Delta$-coloring $\varphi_2 \in \mathcal{C}^\Delta(G - zs_1)$.
		As at least one of $t_1$ and $t_2$ is not $(\beta, \gamma)$-linked with $z$ under $\varphi_2$, we consider the following cases. If $t_1$ is not $(\beta, \gamma)$-linked with $z$, we do a $(\beta, \gamma)$-swap at $t_1$, then recolor $s_1t_1$ by $\eta$ and color $zs_1$ by $\beta$ to get a $ \Delta $-coloring of $ G $, which is impossible. If $t_2$ is not $(\beta, \gamma)$-linked with $z$, we do a $(\beta, \gamma)$-swap at $t_2$. Then $P_z(\beta, \eta) = (z, zs_2, s_2, s_2t_2, t_2)$, showing a contradiction to the fact that $ z $ and $ s_1 $ are $ (\beta, \eta) $-linked.
		
		Overall, the proof of the Lemma is complete.
	\end{proof}
	
	In a $\Delta$-critical graph $G$ with $xy \in E(G)$ and a coloring $\varphi\in\mathcal{C}^\Delta(G-xy)$, a \textit{fork} H (with respect to $\varphi$) is a subgraph defined as follow:
	\begin{center}
		$V(H) = \{x, y, z, s_1, s_2, t_1, t_2\}$ \text{and} $E(G) = \{xy, yz, zs_1, zs_2, s_1t_1, s_2t_2\}$
	\end{center}
	satisfying the conditions: $\varphi(yz) \in \overline{\varphi}(x)$, $\varphi(zs_1), \varphi(zs_2) \in \overline{\varphi}(x) \cup \overline{\varphi}(y)$, $\varphi(s_1t_1) \in ( \overline{\varphi}(x)$ $\cup \overline{\varphi}(y) ) \cap \overline{\varphi}(t_2)$ and $\varphi(s_2t_2) \in$ $( \overline{\varphi}(x) \cup \overline{\varphi}(y) ) \cap \overline{\varphi}(t_1)$. 
	The concept of a fork was introduced in \cite{caoaveragedegree2}, where it was proved (see \cite[Proposition B]{caoaveragedegree2}) that no such fork can exist in a $\Delta$-critical graph when the sum of the degrees of $a$, $t_1$, and $t_2$ is small.
	
	\begin{lem}[\cite{caoaveragedegree2}]\label{fork}
		Let $ G $ be a $ \Delta $-critical graph, $ xy \in E(G) $, and $ \{z, s_1, s_2, t_1, t_2\} \subseteq V(G) $. If $|(\overline{\varphi}(x)\cup \overline{\varphi}(y))\cap\overline{\varphi}(t_1)\cap\overline{\varphi}(t_2)|\geq 3$ for all $ \varphi \in \mathcal{C}^{\Delta}(G - xy) $, then $G$ does not contain the fork on $ \{x, y, z, s_1, s_2, t_1, t_2\} $ with respect to $\varphi$.
	\end{lem}
	
	Note that in [\cite{caoaveragedegree2}, Theorem 2], the condition used was ``$\Delta \geq d(x) + d(t_1) +d(t_2) +1$'', whereas here we use $|(\overline{\varphi}(x)\cup \overline{\varphi}(y))\cap\overline{\varphi}(t_1)\cap\overline{\varphi}(t_2)|\geq 3$ for all $ \varphi \in \mathcal{C}^{\Delta}(G - xy) $. In fact, these two conditions play an equivalent role in the proof of the lemma; one can refer to the Remark 5.1 in [\cite{caoaveragedegree2}, pp. 477].
	
	
	\begin{lem}[\cite{caoovf2022}]\label{ovfdl}
		Let $G$ be an $n$-vertex $\Delta$-critical graph, and let $a \in V(G)$. If $d(a) \leq \frac{2\Delta-n+2}{3}$, then for each $v \in V(G)\backslash\{a\}$, either $d(v) \geq \Delta-d(a)+1$ or $d(v) \leq n-\Delta+2d(a)-6$. Furthermore, if $d(v) \geq \Delta-d(a)+1$, then for any $b \in N(a)$ with $d(b)=\Delta$ and $\varphi \in \mathcal{C}^{\Delta}(G-ab)$, $|\overline{\varphi}(v) \cap (\overline{\varphi}(a) \cup \overline{\varphi}(b))| \leq 1$.
	\end{lem}
	
	We present a refinement of Lemma \ref{ovfdl} as follows.
	
	\begin{lem}\label{ovfdlc}
		Let $G$ be an $n$-vertex $\Delta$-critical graph, and let $a \in V(G)$. If $d(a) \leq \frac{2\Delta-n+5}{3}$, then for each $v \in V(G)\backslash\{a\}$, either $d(v) \geq \Delta-d(a)+1$ or $d(v) \leq n-\Delta+2d(a)-6$. Furthermore, the following holds.
		\begin{itemize}
			\item [(i)]  If $d(v) \geq \Delta-d(a)+1$, then for any $b \in N(a)$ with $d(b)=\Delta$ and $\varphi \in \mathcal{C}^{\Delta}(G-ab)$, $|\overline{\varphi}(v) \cap (\overline{\varphi}(a) \cup \overline{\varphi}(b))| \leq 1$.
			\item [(ii)] If there exist two distinct vertices $v_1, v_2 \in V(G)\backslash\{a\}$ with $d(v_1)=d(v_2)=n-\Delta+2d(a)-6$, then $v_1$ is adjacent to $v_2$.
		\end{itemize}
	\end{lem}
	
	\begin{proof}
		Let $k = d(a)$. Suppose, for contradiction, that there exists $v \in V(G)\backslash\{a\}$ such that $n-\Delta+2k-5 \leq d(v) \leq \Delta-k$. By Lemma \ref{VAL} (VAL), we can choose $b \in N(a)$ with $d(b) = \Delta$ and a coloring $\varphi \in \mathcal{C}^{\Delta}(G-ab)$.
		
		Since $d(b) = \Delta$, we have $|N(b)\backslash\{a\}| = \Delta-1$. Given $d(v) \leq \Delta-k$, VAL implies $a \notin N(v)$. Noting that $b, v \notin N(b) \cap N(v)$ and $d(v) \geq n-\Delta+2k-5$, we derive
		\[
		\begin{aligned}
			|N(v) \cap (N(b)\backslash \{a\})| &= |(N(v) \cap (N(b)\backslash \{a\})) \backslash \{b, v\}| \\
			&\geq |N(v)| + |N(b)\backslash \{a\}| - (n-3) - |N(b) \cap \{v\}| - |N(v) \cap \{b\}| \\
			&\geq 2k-3 - |N(b) \cap \{v\}| - |N(v) \cap \{b\}|.
		\end{aligned}
		\]
		
		Observe that $|[1, \Delta] \backslash (\overline{\varphi}(a) \cup \overline{\varphi}(b))| = k-2$. If $bv \notin E(G)$, there exists a vertex $u \in N(v) \cap (N(b)\backslash\{a\})$ such that $\varphi(bu), \varphi(vu) \in \overline{\varphi}(a) \cup \overline{\varphi}(b)$. If $bv \in E(G)$, then since $d(v) \leq \Delta-k$, Lemma \ref{multifan} implies $\varphi(bv) \notin \overline{\varphi}(a)$, hence $\varphi(bv) \notin \overline{\varphi}(a) \cup \overline{\varphi}(b)$. Therefore, there are at most $2(k-3)$ edges between $\{b, v\}$ and $N(v) \cap (N(b)\backslash\{a\})$ colored with colors from $[1, \Delta] \backslash (\overline{\varphi}(a) \cup \overline{\varphi}(b))$. Thus, we can again find a vertex $u \in N(v) \cap (N(b)\backslash\{a\})$ such that $\varphi(bu), \varphi(vu) \in \overline{\varphi}(a) \cup \overline{\varphi}(b)$.
		
		Consequently, $K=(a, ab, b, bu, u, uv, v)$ forms a Kierstead path. Since $d(v) \leq \Delta-k$, we have $|\overline{\varphi}(v)| \geq k$, which implies
		\[
		|\overline{\varphi}(v) \cap (\overline{\varphi}(a) \cup \overline{\varphi}(b))| \geq 2,
		\]
		contradicting Lemma \ref{Kierstead path} (b).
		
		For $(i)$, since $|N(b)\backslash\{a\}| = \Delta-1$ and $d(v) \geq \Delta-k+1$, we have
		\[
		\begin{aligned}
			|N(v) \cap (N(b)\backslash \{a\})| &= |N(v)| + |N(b)\backslash \{a\}| - |N(v) \cup (N(b)\backslash \{a\})| \\
			&\geq 2 \Delta-k - (n-3 +|N(b) \cap \{v\}| + |N(v) \cap \{b\}| + |N(v) \cap \{a\}|) \\
			&\geq 2 \Delta-k - (n-2) - |N(b) \cap \{v\}| - |N(v) \cap \{b\}| \\
			&\geq 2k-3 - |N(b) \cap \{v\}| - |N(v) \cap \{b\}|.
		\end{aligned}
		\]
		where the last inequality follows from $k = d(a) \leq \frac{2\Delta-n+5}{3}$. The remainder of the proof follows the same reasoning as the first part and is therefore omitted.
		
		For $(ii)$, similar to the first part of the proof, we have for any $v\in V(G)\backslash \{a\}$ with $d(v)=n-\Delta+2k-6$,
		\[
		\begin{aligned}
			|N(v) \cap (N(b)\backslash \{a\})| &= |N(v)| + |N(b)\backslash \{a\}|-|N(v)\cup (N(b)\backslash \{a\})| \\
			&\geq |N(v)| + |N(b)\backslash \{a\}| - (n-3) - |N(b) \cap \{v\}| - |N(v) \cap \{b\}| \\
			&\geq 2k-4 - |N(b) \cap \{v\}| - |N(v) \cap \{b\}|.
		\end{aligned}
		\]
		Since $d(v)=n-\Delta+2k-6<n-\Delta+2k-5\leq \Delta-k$ and $|\overline{\varphi}(a)\cup \overline{\varphi}(b)|=\Delta-k+2$, $|\overline{\varphi}(v)\cap(\overline{\varphi}(a)\cup \overline{\varphi}(b))|\geq 2$. Then $|N(v) \cap (N(b)\backslash \{a\})|\leq 2k-4 - |N(b) \cap \{v\}| - |N(v) \cap \{b\}|$; otherwise, there is a $u \in N(v) \cap (N(b)\backslash\{a\})$ with $\varphi(bu), \varphi(vu) \in \overline{\varphi}(a) \cup \overline{\varphi}(b)$ such that $K=(a, ab, b, bu, u, uv, v)$ forms a Kierstead path, contradicting Lemma \ref{Kierstead path} (b). Hence, $|N(v) \cap (N(b)\backslash \{a\})|= 2k-4 - |N(b) \cap \{v\}| - |N(v) \cap \{b\}|$, which implies that $|N(v)\cup (N(b)\backslash \{a\})|= n-3+|N(b) \cap \{v\}| + |N(v) \cap \{b\}|$. 
		In other words, for any $u\in V(G)\backslash \{a, b, v\}$, we have $u \in N(b)$ or $u\in N(v)$. 
		Moreover, if $v\in N(b)$, then $\varphi(bv)\in\varphi(a)\cap\varphi(b)$ and there are exactly $2(k-3)$ edges between $\{b , v\}$ and $N(v)\cap(N(b)\backslash \{a\})$ colored by colors from $\varphi(a)\cap\varphi(b)$; if $v\notin N(b)$, then there are exactly $2(k-2)$ edges between $\{b , v\}$ and $N(v)\cap(N(b)\backslash \{a\})$ colored by colors from $\varphi(a)\cap\varphi(b)$. 
		Thus, for any $\alpha \in \varphi(a)\cap\varphi(b)$, we have either $b_\alpha \in N(v)\cap N(b)$ or $b_\alpha =v$ where $b_\alpha$ denote the other end vertex of the edge $bb_\alpha$ with $\varphi(bb_\alpha)=\alpha$.
		
		Now, if there exists one of $v_1$ and $v_2$ is not adjacent to $b$, then $v_1$ and $v_2$ are adjacent since for any $i\in\{1, 2\}$ and $u\in V(G)\backslash \{a, b, v_i\}$, we have $u \in N(b)$ or $u\in N(v_i)$. 
		Suppose that $v_1, v_2\in N(b)$. Let $\gamma = \varphi(bv_2) \in \varphi(a) \cap \varphi(b)$. Since $d(v_1)=n-\Delta+2k-6$, we have for any $\alpha \in \varphi(a)\cap\varphi(b)$, either $b_\alpha \in N(v_1)\cap N(b)$ or $b_\alpha =v_1$, which implies that $b_\gamma \in N(v_1) \cap N(b)$ or $b_\gamma = v_1$. Since $b_\gamma = v_2$ and $v_1 \neq v_2$, it follows that $v_2 \in N(v_1) \cap N(b)$.
		Thus, in all cases, $v_1$ and $v_2$ are adjacent.
	\end{proof}


\overfullthm*
	
	We prove this Theorem using the approach developed by \cite{caoovf2022}. For readers convenience, we still give its proof here.
	
	\begin{proof}
		Following the result of Chetwynd and Hilton \cite{ChetwyndHiltonStar1988}, who established the overfull conjecture for graphs with $\Delta \geq n-3$, we restrict our analysis to the case $\Delta \leq n-4$. 
		Then $\Delta(G)\geq\frac{2n+5\delta(G)-12}{3}$ implies that $\delta(G) \leq \frac{n}{5}$. As $\delta(G) \geq 2$, we deduce $n \geq 10$.
		
		Let $ x \in V(G) $ with $ d(x) = \delta(G) = k $. Since $k \geq 2$, $n \geq 10$ and $\Delta(G) \geq \frac{2n+5k-12}{3}$, we have $\frac{2\Delta-n+5}{3} \geq \frac{n+10k-9}{9}=k+\frac{n+k}{9}-1 > k$ and $\Delta>k$. Applying Lemma \ref{ovfdlc}, for every $v \in V(G)\backslash\{x\}$, we have either $d(v) \geq \Delta-k+1$ or $d(v) \leq (n-\Delta)+2k-6$. By Lemma \ref{VAL} (VAL), there exists $ y \in N(x) $ such that $ d(y) = \Delta $. Let $\varphi \in \mathcal{C}^{\Delta}(G-xy)$. We then consider the following two cases.
		
		\textbf{Case 1:} There exist $ t_1, t_2 \in V(G) \backslash \{x\} $ with $ d(t_1), d(t_2) \leq n - \Delta + 2k - 6 $.
		
		Since $ d(x) = k = \delta(G)$ and $|[1, \Delta] \backslash (\overline{\varphi}(x) \cup \overline{\varphi}(y))| = k - 2$, there are at least two edges incident to each of $t_1$ and $t_2$ colored by colors from $\overline{\varphi}(x) \cup \overline{\varphi}(y)$. Thus, there exist two distinct vertices $s_1 \in N(t_1) $ and $s_2 \in N(t_2) $ such that $ \varphi(s_1t_1), \varphi(s_2t_2) \in \overline{\varphi}(x) \cup \overline{\varphi}(y)$. Recall that $G$ is $\Delta$-critical and $s_i \in N(t_i)$. Then it follows from  $\Delta(G)\geq \frac{2n+5k-12}{3}$ and $k \geq 2$ that $d(s_i) \geq \Delta + 2 - (n - \Delta + 2k - 6)> n - \Delta + 2k - 6$. 
		By Lemma \ref{ovfdlc}, we have $d(s_i) \geq \Delta -k+1$ for $i=1, 2$. Thus, since $y \notin N(y)$ and $s_1, s_2 \notin N(s_1) \cap N(s_2)$, we get 
		\begin{equation}\label{neq-1}
			\begin{aligned}
				& |N(y) \cap N(s_1) \cap N(s_2)|\\
				= & |N(y)|+| N(s_1) \cap N(s_2)|-|N(y) \cup (N(s_1) \cap N(s_2))|\\
				\geq & \Delta + | N(s_1) \cap N(s_2)|-\left(n-3+|N(y) \cap \{s_1, s_2\}|+|\{ y\}\cap N(s_1) \cap N(s_2)|\right)\\
				\geq & \Delta + 2(\Delta -k + 1)- n -\left(n-5+|N(y) \cap \{s_1, s_2\}|+|\{x, y, y'\}\cap N(s_1) \cap N(s_2)|\right)\\
				\geq & 3\Delta-2k-2n + 7 -|N(s_1) \cap \{x, y, y'\}|-|N(s_2) \cap \{x, y, y'\}|-|N(y) \cap \{s_1, s_2\}|\\
				\geq & (k-2-|N(s_1) \cap \{x, y, y'\}|) + (k-2-|N(s_2) \cap \{x, y, y'\}|) + \\
				& (k-2-|N(y) \cap \{s_1, s_2\}|) + 1,
			\end{aligned}
		\end{equation}
		where the last inequality follows from $\Delta \geq \frac{2n+5k-12}{3}$. 
		Since $d(s_i) \geq \Delta-k+1>k$ and $d(t_i) \leq n-\Delta+2k-6 \leq \Delta -k$, we have $s_i \neq x$ and $|(\overline{\varphi}(x) \cup \overline{\varphi}(y))\cap\overline{\varphi}(t_i)|\geq 2$ under any $\Delta$-coloring of $G-xy$. 		  
		Therefore, if $s_iy \in E(G)$, then $\varphi(s_iy) \notin \overline{\varphi}(x)$ by Lemma \ref{Kierstead path} $(b)$, which implies that $\varphi(s_iy) \notin \overline{\varphi}(x) \cup \overline{\varphi}(y)$. 
		If $s_ix \in E(G)$, then $\varphi(s_ix) \notin \overline{\varphi}(y)$ by Lemma \ref{multifan}, which implies that $\varphi(s_ix) \notin \overline{\varphi}(x) \cup \overline{\varphi}(y)$. 
		Let $y' \in N(x)$ with $\varphi(xy') = \overline{\varphi}(y)$. By considering the coloring $\varphi_1$ obtained from $\varphi$ by uncoloring $xy'$ and coloring $xy$ by $\varphi(xy')$, we have $\overline{\varphi}(x)\cup \overline{\varphi}(y)\subseteq\overline{\varphi_1}(x)\cup\overline{\varphi_1}(y')$ due to $|\overline{\varphi}(y)|=1$, and so
		if $s_iy' \in E(G)$, then $\varphi(s_iy') \notin \overline{\varphi}(x)$ by Lemma \ref{Kierstead path} $(b)$, which implies that  $\varphi(s_iy') \notin \overline{\varphi}(x) \cup \overline{\varphi}(y)$.
		Since $|[1, \Delta] \backslash (\overline{\varphi}(x) \cup \overline{\varphi}(y))|=k-2$, the inequality (\ref{neq-1}) implies that there exists $z \in N(y) \cap N(s_1) \cap N(s_2)$ such that $\varphi(yz), \varphi(zs_1), \varphi(zs_2) \in \overline{\varphi}(x) \cup \overline{\varphi}(y)$. Since both $\{x, y, t_i\}$ and $\{x, y, s_i, t_i\}$ are not elementary, we obtain that $z \notin \{x, t_1, t_2\}$ by $\varphi(s_it_i) \in \overline{\varphi}(x) \cup \overline{\varphi}(y)$ (for $i \in [1,2]$), Lemmas \ref{multifan} and \ref{Kierstead path}.
		Thus, $K=(x, xy, y, yz, z,$ $zs_1, s_1, s_1t_1, t_1)$ and $K^*=(x, xy, y, yz, z, zs_2, s_2, s_2t_2, t_2)$ are two Kierstead paths.
		
		If there exist $ s_1 \in N(t_1) $ and $ s_2 \in N(t_2) $ such that $ \varphi(s_1t_1), \varphi(s_2t_2) \in \overline{\varphi}(x) \cup \overline{\varphi}(y) $, and $ \varphi(s_1t_1) = \varphi(s_2t_2) $, then let $ \Gamma = \overline{\varphi}(t_1) \cap \overline{\varphi}(t_2) $. 
		Since $ |[1, \Delta] \backslash (\overline{\varphi}(x) \cup \overline{\varphi}(y))| = k - 2 $, we have $ |\Gamma \cap (\overline{\varphi}(x) \cup \overline{\varphi}(y))| \geq |\Gamma| - k + 2 $. 
		If $d(t_1) = d(t_2)=n-\Delta+2k-6$, then by Lemma \ref{ovfdlc} $(ii)$, $t_1$ is adjacent to $t_2$ which implies that $|\varphi(t_1)\cap\varphi(t_2)|\geq |\varphi(s_1t_1)|+|\varphi(t_1t_2)|=2$.
		Then, by the degree condition of $t_1, t_2$, we obtain $ |\Gamma| = \Delta-|\varphi(t_1)\cup\varphi(t_2)| \geq \Delta - 2(n - \Delta + 2k - 6) + 2 = 3\Delta - 2n - 4k + 14$. 
		Given $\Delta \geq \frac{2n+5k-12}{3}$, we have:
		\[
		|\Gamma \cap (\overline{\varphi}(x) \cup \overline{\varphi}(y))| \geq |\Gamma| - k + 2 \geq 3\Delta - 2n - 5k + 16= 4 > 3,
		\]
		which contradicts Lemma \ref{lfork}.
		
		Thus, for any $ s_1 \in N(t_1) $ and $ s_2 \in N(t_2) $ with $ \varphi(s_1t_1), \varphi(s_2t_2) \in \overline{\varphi}(x) \cup \overline{\varphi}(y) $, it must hold that $ \varphi(s_1t_1) \neq \varphi(s_2t_2) $. Consequently, $ \varphi(s_1t_1) \in \overline{\varphi}(t_2) $ and $ \varphi(s_2t_2) \in \overline{\varphi}(t_1) $. Hence, the subgraph $ H^* $ with $ V(H^*) = \{x, y, z, s_1, s_2, t_1, t_2\} $ is a fork. Again, by Lemma \ref{ovfdlc} $(ii)$, we have $|\varphi(t_1)\cap\varphi(t_2)|\geq |\varphi(t_1t_2)|=1$, and
		\[
		\begin{aligned}
		&|(\overline{\varphi}(x)\cup \overline{\varphi}(y))\cap\overline{\varphi}(t_1)\cap\overline{\varphi}(t_2)|\\
		\geq & |\overline{\varphi}(x)\cup \overline{\varphi}(y)|-(\Delta-|\overline{\varphi}(t_1)\cap\overline{\varphi}(t_2)|)\\
		= & |\overline{\varphi}(x)\cup \overline{\varphi}(y)|-|\varphi(t_1)\cup \varphi(t_2)|\\
		= & |\overline{\varphi}(x)\cup \overline{\varphi}(y)|-|\varphi(t_1)|-|\varphi(t_2)|+|\varphi(t_1)\cap \varphi(t_2)|\\
		\geq & \Delta -k+2-2(n-\Delta+2k-6)+1\\
		\geq & 3,
		\end{aligned}
		\]
		where the last inequality follows from $\Delta \geq \frac{2n+5k-12}{3}$, contradicting Lemma \ref{fork}.
		
		\textbf{Case 2:} There exists at most one $ t \in V(G) \backslash \{x\} $ with $ d(t) \leq n - \Delta + 2k - 6 $.
		
		Suppose to the contrary that $G$ is not overfull. Then there exist $u, v \in V(G)\backslash \{x\}$ such that $$
		\Delta-k+1 \leq d(u), d(v)<\Delta,\ \overline{\varphi}(u) \cap (\overline{\varphi}(x) \cup \overline{\varphi}(y)) \neq \emptyset,\ \overline{\varphi}(v) \cap (\overline{\varphi}(x) \cup \overline{\varphi}(y)) \neq \emptyset\ (\ast\ast)
		$$(see [\cite{caoovf2022}, p. 2265] for existence verification). For the reader's convenience, a detailed verification of this fact is provided in Claim \ref{u v exit} at the end of this proof.

		Note that for any $w \in \{u, v\}$, by $\overline{\varphi}(w) \cap (\overline{\varphi}(x) \cup \overline{\varphi}(y)) \neq \emptyset$ and Lemma \ref{Kierstead path}, we have if $yw \in E(G)$, then $\varphi(yw) \notin \overline{\varphi}(x) \cup \overline{\varphi}(y)$. 
		If $wx \in E(G)$, then $\varphi(wx) \notin \overline{\varphi}(y)$ by Lemma \ref{multifan}, which implies that $\varphi(wx) \notin \overline{\varphi}(x) \cup \overline{\varphi}(y)$. 
		Recall that $d(x)=k<\Delta=d(y)$ and $\overline{\varphi}(w) \cap (\overline{\varphi}(x) \cup \overline{\varphi}(y)) \neq \emptyset$. Let $y' \in N(x)$ with $\varphi(xy') = \overline{\varphi}(y)$.
		If $wy' \in E(G)$, then $\varphi(wy') \notin \overline{\varphi}(x)$ by Lemma \ref{Kierstead path} $(a)$, which implies  that $\varphi(wy') \notin \overline{\varphi}(x) \cup \overline{\varphi}(y)$. Since $y, u, v \notin \{N(y) \cap N(u) \cap N(v)\}$, we have
		\[
		\begin{aligned}
			& |N(y) \cap N(u) \cap N(v)|\\
			\geq & |N(y)|+|N(u) \cap N(v)|-(n-3+|\{y\} \cap N(u) \cap N(v)|+|\{u, v\} \cap N(y)|)\\
			\geq & d(u)+d(v)-2n+\Delta+5-|N(u) \cap \{x, y, y'\}|-|N(v) \cap \{x, y, y'\}|-|N(y) \cap \{u, v\}|\\
			\geq & 3\Delta-2k-2n+7-|N(u) \cap \{x, y, y'\}|-|N(v) \cap \{x, y, y'\}|-|N(y) \cap \{u, v\}|\\
			\geq & (k-2-|N(u) \cap \{x, y, y'\}|) + (k-2-|N(v) \cap \{x, y, y'\}|) + \\
			& (k-2-|N(y) \cap \{u,v\}|)+1,
		\end{aligned}
		\]
		where the last inequality follows from $\Delta \geq \frac{2n+5k-12}{3}$. By the inequality above and the same argument as in Case 1, we can find $z \in N(y) \cap N(u) \cap N(v)$ such that both 
		\[
		K = (x, xy, y, yz, z, zu, u) \ \mbox{and}\ K^* = (x, xy, y, yz, z, zv, v)
		\]
		are Kierstead paths with respect to $xy$ and $\varphi$. Recall that  $\max\{d(u), d(v)\} < \Delta$. So we get a contradiction to Lemma \ref{tkpl}.
		
		\begin{claim}\label{u v exit}[\cite{caoovf2022}, p. 2265]
			If $G$ is not overfull, then there exist two distinct vertices $u, v \in V(G) \backslash \{x\}$ such that $\Delta-k+1 \leq d(u), d(v)<\Delta$, $\overline{\varphi}(u) \cap (\overline{\varphi}(x) \cup \overline{\varphi}(y)) \neq \emptyset$ and $\overline{\varphi}(v) \cap (\overline{\varphi}(x) \cup \overline{\varphi}(y)) \neq \emptyset$.
		\end{claim}
		
		\begin{proof}
			For completeness, we provide the detailed argument from \cite{caoovf2022} here.
			
			Assume first that $n$ is odd.
			As $G$ is not overfull, some color in $[1, \Delta]$ is missing at least three distinct vertices of $G$ by the Parity Lemma. If there exists exactly one vertex $t \in V(G)\backslash\{x\}$ such that $d(t) \leq n-\Delta+2k-6$, then each color in $\overline{\varphi}(t) \cap (\overline{\varphi}(x) \cup \overline{\varphi}(y))$ is missing at another vertex from $V(G)\backslash\{x, y, t\}$. As $|\overline{\varphi}(t) \cap (\overline{\varphi}(x) \cup \overline{\varphi}(y))| \geq 2\Delta-n-2k+6-(k-2)=2\Delta-n-3k+8>k$ and $d(v) \geq \Delta-k+1$ for every $v \in V(G)\backslash\{x, y, t\}$, we can find $u, v \in V(G)\backslash\{x\}$ with the desired property in ($\ast\ast$). Thus for every vertex $v \in V(G) \backslash \{x, y\}$, $d(v) \geq \Delta - k + 1$. Let $\beta \in [1, \Delta]$ such that $\beta$ is missing at at least three vertices, say $a, b, c$, from $V(G)$. If $\beta \in \overline{\varphi}(x) \cup \overline{\varphi}(y)$, as $\overline{\varphi}(x) \cup \overline{\varphi}(y)$ is $\varphi$-elementary, then letting $u, v \in \{a, b, c\} \backslash \{x, y\}$ will give us a desired choice. So we assume $\beta \notin \overline{\varphi}(x) \cup \overline{\varphi}(y)$. Let $\alpha \in \overline{\varphi}(x) \cup \overline{\varphi}(y)$, say, w.l.o.g., that $\alpha \in \overline{\varphi}(x)$. As at most one of $a, b, c$, say $c$ is $(\alpha, \beta)$-linked with $a$, we do an $(\alpha, \beta)$-swap at the other two vertices $a$ and $b$. Call the new coloring still $\varphi$. Now $a$ and $b$ can play the role of $u$ and $v$.
			
			Assume now $n$ is even.
			As $G$ is not overfull, each color in $[1, \Delta]$ is missing at an even number of vertices of $G$ by the Parity Lemma. In particular, we have fact (a): each color in $\overline{\varphi}(x) \cup \overline{\varphi}(y)$ is missing at a vertex from $V(G)\backslash\{x, y\}$; fact (b): for every $v \in V(G)\backslash\{x, y\}$, if $\Delta-k+1 \leq d(v) \leq \Delta-1$, then $|\overline{\varphi}(v) \cap (\overline{\varphi}(x) \cup \overline{\varphi}(y))| \leq 1$ (by the second part of Lemma \ref{ovfdl}).
			Let $t \in V(G)\backslash\{x, y\}$ such that $d(t)$ is the smallest among the degrees of vertices from $V(G)\backslash\{x, y\}$. As $|\overline{\varphi}(x) \cup \overline{\varphi}(y)|=\Delta-k+2$ and $|\overline{\varphi}(t)| \leq \Delta-k$, there exist at least two distinct colors $\alpha, \beta \in (\overline{\varphi}(x) \cup \overline{\varphi}(y))\backslash\overline{\varphi}(t)$. By the assumption of this case and Lemma \ref{ovfdl}, we know that every vertex from $V(G)\backslash\{x, y, t\}$ has degree at least $\Delta-k+1$. Thus by facts (a) and (b), we can find $u, v\in V(G)\backslash\{x, y, t\}$ such that $\alpha \in \overline{\varphi}(u)$ and $\beta \in \overline{\varphi}(v)$. Clearly, $u$ and $v$ satisfy the property in ($\ast\ast$).
		\end{proof}
		Overall, the proof of the Theorem is complete.
	\end{proof}

\bibliography{referencesofthefirstpaper}
	
\end{document}